\documentclass{amsart}

\usepackage[]{fontenc}
\usepackage[latin1]{inputenc}
\usepackage{geometry}
\usepackage{fancyhdr}
\usepackage{setspace}
\usepackage{pstricks}
\usepackage{pst-node}
\usepackage{pst-plot}
\usepackage[all]{xy}
\usepackage{amssymb}

\usepackage{amsfonts}
\usepackage{eufrak}
\usepackage{amsmath}
\usepackage{amsthm}
\usepackage{a4wide}
\usepackage{textcomp}
\usepackage{epsfig}

\newtheorem{thm}{Theorem}[section]
\newtheorem{defi}[thm]{Definition}
\newtheorem{cor}[thm]{Corollary}
\newtheorem{q}[thm]{Question}
\newtheorem{prop}[thm]{Proposition}
\newtheorem{rmk}[thm]{Remark}
\newtheorem{lemma}[thm]{Lemma}

\usepackage[]{fontenc}
\usepackage[latin1]{inputenc}
\usepackage{geometry}
\usepackage{fancyhdr}
\usepackage{setspace}
\usepackage{pstricks}
\usepackage{pst-node}
\usepackage{pst-plot}
\usepackage[all]{xy}
\usepackage{amssymb}

\usepackage{amsfonts}
\usepackage{eufrak}
\usepackage{amsmath}
\usepackage{amsthm}
\usepackage{a4wide}
\usepackage{textcomp}
\usepackage{epsfig}

\newcommand{\R}{\mathbb{R}}
\newcommand{\Z}{\mathbb{Z}}
\newcommand{\N}{\mathbb{N}}

\begin{document}

\title[The discriminant and oscillation lengths]{The discriminant and oscillation lengths for contact and Legendrian isotopies}

\author{Vincent Colin}
\address{Universit\'e de Nantes, 44322 Nantes, France}
\email{vincent.colin@univ-nantes.fr}

\author{Sheila Sandon}
\address{Universit\'e de Nantes and CNRS, 44322 Nantes, France}
\email{sheila.sandon@univ-nantes.fr}

\thanks{V. Colin supported by the Institut Universitaire de France, ANR Floer Power and ERC Geodycon. S. Sandon supported by ANR Floer Power, ERC Geodycon and National Science Foundation under agreement No. DMS-0635607.}

\begin{abstract}
\noindent
We define an integer-valued non-degenerate bi-invariant metric (the \emph{discriminant metric}) on the universal cover of the identity component of the contactomorphism group of any contact manifold. This metric has a very simple geometric definition, based on the notion of discriminant points of contactomorphisms. Using generating functions we prove that the discriminant metric is unbounded for the standard contact structures on $\mathbb{R}^{2n}\times S^1$ and $\mathbb{R}P^{2n+1}$. On the other hand we also show by elementary arguments that the discriminant metric is bounded for the standard contact structures on $\mathbb{R}^{2n+1}$ and $S^{2n+1}$. As an application of these results we get that the contact fragmentation norm is unbounded for $\mathbb{R}^{2n}\times S^1$ and $\mathbb{R}P^{2n+1}$. By elaborating on the construction of the discriminant metric we then define an integer-valued bi-invariant pseudo-metric, that we call the \emph{oscillation pseudo-metric}, which is non-degenerate if and only if the contact manifold is orderable in the sense of Eliashberg and Polterovich and, in this case, it is compatible with the partial order. Finally we define the discriminant and oscillation lengths of a Legendrian isotopy, and prove that they are unbounded for $T^{\ast}B\times S^1$ for any closed manifold $B$, for $\mathbb{R}P^{2n+1}$ and for some $3$-dimensional circle bundles.
\end{abstract}

\maketitle

\section{Introduction}\label{section: introduction}

Since its discovery in 1990, the Hofer metric on the Hamiltonian group of a symplectic manifold \cite{H} has been considered one of the most important notions in symplectic topology. It is a manifestation of symplectic rigidity, and is crucially related to other symplectic rigidity phenomena such as the existence of symplectic capacities and Gromov's non-squeezing theorem.

In 2000 Eliashberg and Polterovich \cite{EP} noticed that there can be no analogue of the Hofer metric on the contactomorphism group, and raised the question of whether the contactomorphism group admits any geometric structure at all. Motivated by this question, they introduced in their paper the notion of \textit{orderability} of contact manifolds: a contact manifold is said to be orderable if the natural relation induced by positive contact isotopies (i.e. contact isotopies that move every point in a direction positively transverse to the contact distribution) gives a bi-invariant partial order on the universal cover of the contactomorphism group. The original motivation of Eliashberg and Polterovich for introducing this notion was that if $(M,\xi)$ is an orderable  contact manifold then we can apply a general  procedure, that works for any partially ordered group, to associate to the universal cover of the contactomorphism group of $(M,\xi)$ a partially ordered metric space $\big(Z(M,\xi),\delta\big)$. This is done by first defining, in terms of the relative growth, a pseudo-distance $\delta$ on the group of elements of the universal cover that are generated by a positive contact isotopy, and then by quotienting this group by the equivalence classes of elements which are at zero distance from each other. After the work of Eliashberg and Polterovich, \textit{integer-valued} bi-invariant metrics on the contactomorphism group itself have been discovered in the case of $\mathbb{R}^{2n}\times S^1$ by the second author \cite{mio2}, in the case of $T^{\ast}B\times S^1$ for any closed manifold $B$ by Zapolsky \cite{Z}, and for some other classes of circle bundles by Fraser, Polterovich and Rosen \cite{FPR} and Albers and Merry \cite{AM}.

In this paper we define an integer-valued non-degenerate bi-invariant metric on the universal cover of the identity component of the contactomorphism group of any contact manifold. We call this metric the \textbf{discriminant metric}. It is the metric which is associated to the norm\footnote{Recall that any conjugation-invariant norm $\nu: G \rightarrow [0,+\infty)$ on a group $G$ induces a bi-invariant metric $d_{\nu}$ on $G$ by defining $d_{\nu}(f,g) = \nu(f^{-1}g)$.} described in the following definition.

\begin{defi}\label{definition: definition in the introduction}
Let $\big(M,\xi=\text{ker}(\alpha)\big)$ be a (co-oriented) compact\footnote{If $M$ is not compact, we will consider the discriminant norm on the universal cover of the identity component of the group of compactly supported contactomorphisms.} contact manifold, and consider the universal cover $\widetilde{\text{Cont}_0}(M,\xi)$ of the identity component of the contactomorphism group. The \emph{discriminant norm} of an element of $\widetilde{\text{Cont}_0}(M,\xi)$ is the minimal integer $N$ needed to write a representative $\{\phi_t\}_{t\in[0,1]}$ as the concatenation of a finite number of pieces $\{\phi_t\}_{t\in[t_i,t_{i+1}]}$, $i=0,\cdots,N-1$, such that, for each $i$, the submanifold $\bigcup_{t\in[t_i,t_{i+1}]}\text{gr}(\phi_t)$ of $M\times M\times \mathbb{R}$ is embedded.
\end{defi}

In other words, the discriminant metric is the word metric on $\widetilde{\text{Cont}_0}(M,\xi)$ with respect to the generating set formed by (non-constant) contact isotopies $\{\phi_t\}_{t\in[0,1]}$ such that $\bigcup_{t\in[0,1]}\text{gr}(\phi_t)$ is embedded (see Section \ref{section: discriminant metric}). Recall that the graph of a contactomorphism $\phi$ of $M$ is the Legendrian submanifold of the contact product $M\times M\times\mathbb{R}$ which is defined by $\text{gr}(\phi) = \{\,\big(q,\phi(q),g(q)\big)\,|\,q\in M\,\}$, where $g:M\rightarrow\mathbb{R}$ is the function satisfying $\phi^{\ast}\alpha=e^g\alpha$. In Section \ref{section: discriminant metric} we show, by elementary arguments, that Definition \ref{definition: definition in the introduction} makes sense, does not depend on the choice of the contact form, and always gives rise to a non-degenerate bi-invariant metric on $\widetilde{\text{Cont}_0}(M,\xi)$.

Note that an analogous definition would not work in the symplectic case: for any Hamiltonian isotopy $\{\varphi_t\}$ of a compact symplectic manifold $(W,\omega)$, the union $\bigcup_t\text{gr}(\varphi_t)$ in $\overline{W}\times W$ will never be embedded, because of the Arnold conjecture.

To understand the geometric meaning of the discriminant metric, and the motivation for its name, note that a union $\bigcup_t\text{gr}(\phi_t)$ fails to be embedded if for some values $t_0$ and $t_1$ the graph of $\phi_{t_0}$ intersects the graph of $\phi_{t_1}$. Equivalently, this happens if and only if $\phi_{t_0}^{\phantom{t}-1} \circ \phi_{t_1}$ has a \textit{discriminant point}.  Recall that a point $q$ of $M$ is called a discriminant point of a contactomorphism $\phi$ if $\phi(q)=q$ and $g(q)=0$, where $g:M\rightarrow\mathbb{R}$ is the function satisfying $\phi^{\ast}\alpha=e^g\alpha$. This notion is independent of the choice of contact form, and was introduced by Givental in \cite{Giv}. Recall also that a point $q$ of $M$ is called a \textit{translated point} of $\phi$ (with respect to a chosen contact form) if $q$ and $\phi(q)$ belong to the same Reeb orbit, and $g(q)=0$. Thus, discriminant points are translated points that are also fixed points. As it is discussed in \cite{mio5}, translated points seem to satisfy an analogue of the Arnold conjecture. At least in the case of a $\mathcal{C}^0$-small contact isotopy $\{\phi_t\}$ (but possibly also in the general case) every $\phi_t$ always has translated points (with respect to any contact form), at least as many as the minimal number of critical points of a function on $M$. Note however that the contactomorphisms $\phi_t$ do not necessarily have discriminant points. On the other hand, discriminant points, in contrast to translated points, are invariant by conjugation: if $q$ is a discriminant point of $\phi$ then $\psi(q)$ is a discriminant point of $\psi\phi\psi^{-1}$, for any other contactomorphism $\psi$. Thus, while in the symplectic case fixed points at the same time persist under Hamiltonian isotopies (because of the Arnold conjecture) and are invariant by conjugation,  in the contact case these two phenomena split: translated points persist by contact isotopy but are not invariant by conjugation, while discriminant points are invariant by conjugation but do not persist under contact isotopy. Roughly speaking, this is what gives room for the definition of the discriminant metric to work in the contact case: since discriminant points do not persist under contact isotopy, Definition \ref{definition: definition in the introduction} makes sense (see Lemma \ref{lemma: decomposition}). On the other hand, the fact that discriminant points are invariant by conjugation implies that the discriminant metric is bi-invariant.

In contrast to fixed points of contactomorphisms (that are completely flexible), translated and discriminant points turn out to be related to interesting contact rigidity phenomena. For example, translated and discriminant points for contactomorphisms of $\mathbb{R}^{2n}\times S^1$ play a crucial role in the proof of the contact non-squeezing theorem \cite{EKP, mio1}, while translated and discriminant points of contactomorphisms of $\mathbb{R}P^{2n+1}$ are intimately related to Givental's non-linear Maslov index \cite{Giv}. In view of these examples we believe that the discriminant metric should be an interesting object to study, since its properties on a given contact manifold might reflect the existence of contact rigidity phenomena such as contact non-squeezing, orderability, or the existence of quasimorphisms on the contactomorphism group.

An important question about the discriminant metric is to understand for which contact manifolds it is unbounded. In this paper we prove the following results.

\begin{thm}\label{theorem: main in introduction}
\begin{enumerate}
\renewcommand{\labelenumi}{(\roman{enumi})}
\item The discriminant metric is bounded for the standard contact structures on the Euclidean space $\mathbb{R}^{2n+1}$ and on the sphere $S^{2n+1}$.
\item The discriminant metric is unbounded for the standard contact structure on $\mathbb{R}^{2n}\times S^1$.
\item The discriminant metric is unbounded for the standard contact structure on the projective space $\mathbb{R}P^{2n+1}$.
\end{enumerate}
\end{thm}

Note that the distinction between boundedness and unboundedness is especially important when dealing with integer-valued metrics. Indeed, if an integer-valued metric is bounded then it is equivalent to the trivial metric, i.e. the metric for which any two distinct points are always at distance one from each other. The notion of equivalence we use here is the one introduced by Burago, Ivanov and Polterovich \cite{BIP}: given two bi-invariant metrics on the same group we call them equivalent if their ratio is bounded away from zero and from infinity. Note also that boundedness of the discriminant metric for $S^{2n+1}$ is consistent with the fact, proved by Fraser, Polterovich and Rosen \cite{FPR}, that any non-degenerate bi-invariant metric on the standard contact sphere is necessarily equivalent to the trivial metric.

Boundedness for $\mathbb{R}^{2n+1}$ and $S^{2n+1}$ is proved in Section \ref{section: boundedness} by elementary arguments. Unboundedness for $\mathbb{R}^{2n}\times S^1$ and $\mathbb{R}P^{2n+1}$ is proved in Sections \ref{section: RxS} and \ref{section: projective space} respectively, using generating functions.

The difference between the cases of $\mathbb{R}^{2n+1}$ and $\mathbb{R}^{2n}\times S^1$ is similar to the phenomenon discussed in \cite{mio1}. While in $\mathbb{R}^{2n+1}$ we do not have any contact rigidity result (except for the existence of a partial order on the contactomorphism group \cite{B}), 1-periodicity of the Reeb flow in $\mathbb{R}^{2n}\times S^1$ makes it possible to construct \textit{integer-valued} spectral invariants for contactomorphisms and, using them, an integer-valued contact capacity for domains. The spectral invariants and the contact capacity can then be used to define an integer-valued bi-invariant metric on the contactomorphism group of $\mathbb{R}^{2n}\times S^1$ \cite{mio2} and to prove the contact non-squeezing theorem. Unboundedness of the discriminant metric on $\mathbb{R}^{2n}\times S^1$ is also proved using the spectral invariants defined in \cite{mio1} and thus it relies crucially on the 1-periodicity of the Reeb flow. This might suggest that the discriminant metric should always be unbounded whenever there is a 1-periodic Reeb flow, but the case of $S^{2n+1}$, where the discriminant metric is bounded, shows that this is not true in general. On the other hand, the discriminant metric is unbounded on $\mathbb{R}P^{2n+1}$: this is proved using generating functions in the setting developed by Givental \cite{Giv}, and relies on the properties of the cohomological length of subsets of projective spaces. Note that the same contrast between the cases of $S^{2n+1}$ and $\mathbb{R}P^{2n+1}$ also appears in the context of orderability: $\mathbb{R}P^{2n+1}$ is orderable, while $S^{2n+1}$ is not. Moreover, also the construction of the non-linear Maslov index, which gives a quasimorphism on the universal cover of the contactomorphism group, works for $\mathbb{R}P^{2n+1}$ but not for $S^{2n+1}$. To summarize, the difference between $\mathbb{R}^{2n+1}$ and $\mathbb{R}^{2n}\times S^1$ on the one hand and between $S^{2n+1}$ and $\mathbb{R}P^{2n+1}$ on the other hand shows that, in both cases of $\mathbb{R}^{2n+1}$ and $S^{2n+1}$, adding some topology to the underlying manifold has the effect of making the discriminant metric unbounded and of making some interesting contact rigidity phenomena appear. This interplay between contact rigidity phenomena and the topology of the underlying manifold is quite mysterious and, we believe, an interesting subject for further research.

Using the fact that the discriminant metric is bounded on $\mathbb{R}^{2n+1}$, we observe in Section \ref{section: fragmentation} that the discriminant metric always gives a lower bound for the contact fragmentation norm. Recall that the fragmentation norm on the diffeomorphism group of a smooth compact manifold $M$ is defined as follows. By the fragmentation lemma \cite{Ban} every diffeomorphism $\phi$ of $M$ can be written as a finite composition $\phi=\phi_1\circ\cdots\circ\phi_{N}$ such that each factor is supported in an embedded ball. The fragmentation norm of $\phi$ is then defined to be the minimal number of factors in such a decomposition. By results of Burago, Ivanov and Polterovich \cite{BIP} and Tsuboi \cite{Ts} we know that for any odd-dimensional compact manifold $M$ the smooth fragmentation norm is bounded. The fragmentation lemma has an analogue also for Hamiltonian and contact diffeomorphisms (see \cite{Ban} and \cite{Ry} respectively). While the Hamiltonian fragmentation norm has been studied in the special case of tori by Entov and Polterovich \cite{EnP} and Burago, Ivanov and Polterovich \cite{BIP} and for cotangent bundles by Monzner, Vichery and Zapolsky \cite{MVZ}, we did not find any result in the literature on the contact fragmentation norm. As a corollary of Theorem \ref{theorem: main in introduction} we get the following result.

\begin{cor}\label{corollary: corollary in introduction}
The contact fragmentation norm is unbounded for the standard contact structures on $\mathbb{R}^{2n}\times S^1$ and $\mathbb{R}P^{2n+1}$.
\end{cor}

If in Definition \ref{definition: definition in the introduction} we also require that every embedded piece $\{\phi_t\}_{t\in[t_i,t_{i+1}]}$ be either positive or negative then we get another bi-invariant metric (that we call the \textit{zigzag metric}) on $\widetilde{\text{Cont}_0}(M,\xi)$. We suspect that this metric should be always equivalent, in the sense of \cite{BIP}, to the discriminant metric. A more interesting variation of the discriminant metric is what we call the \textbf{oscillation pseudo-metric}. The oscillation pseudo-norm of an element $[\{\phi_t\}]$ of $\widetilde{\text{Cont}_0}(M,\xi)$ is defined to be the sum of the minimal number of positive pieces and the minimal number of negative pieces in a decomposition of a representative $\{\phi_t\}$ as a concatenation of positive and negative embedded pieces. The oscillation pseudo-norm of an element of $\widetilde{\text{Cont}_0}(M,\xi)$ is always smaller than or equal to, but in general different from, its zigzag norm. The oscillation pseudo-metric is well-defined and bi-invariant for every contact manifold $(M,\xi)$, and it is non-degenerate if and only if $(M,\xi)$ is orderable. In this case it is compatible with the partial order and hence it gives to $\widetilde{\text{Cont}_0}(M,\xi)$ the structure of a partially ordered metric space (see the discussion in Section \ref{section: oscillation}).

In the last section we study the \textbf{Legendrian discriminant length}. Given a Legendrian isotopy $\{L_t\}_{t\in[0,1]}$ in a contact manifold $M$, its discriminant length is the minimal number $N$ needed to write $\{L_t\}_{t\in[0,1]}$ as a concatenation of pieces $\{L_t\}_{t\in[t_i,t_{i+1}]}$, $i=0,\cdots,N-1$ such that each piece is embedded. As for contact isotopies we can then also define the zigzag and the oscillation lengths of $\{L_t\}_{t\in[0,1]}$. As we show in Section \ref{section: Legendrian}, the generating functions methods used in Sections \ref{section: RxS} and \ref{section: projective space} can also be applied to prove unboundedness of the Legendrian discriminant length on $T^{\ast}B \times S^1$ for any compact smooth manifold $B$, on $\mathbb{R}P^{2n+1}$ and on certain $3$-dimensional circle bundles.

The article is organized as follows. In the next two sections we will discuss in more details the construction of the discriminant, zigzag and oscillation (pseudo-)norms on the universal cover of the contactomorphism group. In Section \ref{section: boundedness} we prove that these (pseudo-)norms are bounded on $\mathbb{R}^{2n+1}$ and $S^{2n+1}$, and in Section \ref{section: fragmentation} we use boundedness on $\mathbb{R}^{2n+1}$ to show that the contact fragmentation norm is always bounded below by twice the zigzag norm. In Sections \ref{section: RxS} and \ref{section: projective space} respectively we prove that the discriminant, zigzag and oscillation norms are unbounded for $\mathbb{R}^{2n}\times S^1$ and $\mathbb{R}P^{2n+1}$. As a consequence, the contact fragmentation norm is also unbounded in these two cases. In the last section we discuss the Legendrian discriminant length.

\subsection*{Acknowledgments} We thank Emmanuel Giroux for discussions and support. We also thank Maia Fraser and Leonid Polterovich for their feedback to our work and for giving us a first draft  of \cite{FPR}. Moreover we thank them, and all the participants of the Workshop \textquotedblleft Contact topology in higher dimensions\textquotedblright at the American Institute of Mathematics on May 2012, for all the exciting discussions during that week. We also would like to thank the referee for many very useful comments (and in particular for suggesting us to reformulate our definitions and results by putting them in the context of word metrics). Finally, the second author would like to thank the Institute for Advanced Study in Princeton for the very pleasant atmosphere and working conditions during her visit in the spring term of 2012.

\section{The discriminant metric}\label{section: discriminant metric}

Let $\big(M,\xi=\text{ker}(\alpha)\big)$ be a (cooriented) contact manifold.

We will define in this section a bi-invariant metric on the universal cover $\widetilde{\text{Cont}_0^{\phantom{0}c}}(M,\xi)$ of the identity component of the group of compactly supported contactomorphisms of $(M,\xi)$. To simplify the exposition, in the following discussion we will just assume that $M$ is compact. In the case of compactly supported contactomorphisms of a non-compact contact manifold all arguments and definitions are analogous, by considering only the interior of the support.

Recall that we can associate to any contactomorphism $\phi$ of $M$ its graph $\text{gr}(\phi)$ in the contact product $M\times M\times \mathbb{R}$. This is defined to be the submanifold $\text{gr}(\phi)=\{\,\big(q,\phi(q),g(q)\big)\,|\,q\in M\,\}$ of $M\times M\times \mathbb{R}$,  where $g:M\rightarrow\mathbb{R}$ is the function satisfying $\phi^{\ast}\alpha=e^g\alpha$. The contact structure on $M\times M\times \mathbb{R}$ is given by the kernel of the 1-form $A=e^{\theta}\alpha_1-\alpha_2$ where $\alpha_1$ and $\alpha_2$ are the pullback of $\alpha$ with respect to the projections of $M\times M\times \mathbb{R}$ onto the first and second factors respectively, and $\theta$ is the $\mathbb{R}$-coordinate.

\begin{lemma}\label{lemma: decomposition}
Let $\{\phi_t\}_{t\in [0,1]}$ be a contact isotopy of $M$. After perturbing $\{\phi_t\}$ in the same homotopy class with fixed endpoints, there exist a positive integer $N$ and a subdivision $0=t_0<t_1<\cdots<t_{N-1}<t_N=1$ such that for all $i=0,\cdots,N-1$ the submanifold $\bigcup_{t\in[t_i,t_{i+1}]}\text{gr}(\phi_t)$ of $M\times M\times \mathbb{R}$ is embedded.
\end{lemma}

In the following, we will say that a contact isotopy $\{\phi_t\}$ is \textit{embedded} if the submanifold $\bigcup_t\text{gr}(\phi_t)$ of $M\times M\times \mathbb{R}$ is embedded. Recall also that a contact isotopy is said to be \textit{positive} (respectively \textit{negative}) if it moves every point in a direction positively (negatively) transverse to the contact distribution, or equivalently if it is generated by a positive (negative) contact Hamiltonian. Moreover, a Legendrian isotopy $\{L_t\}$ is said to be positive (negative) if it can be extended to a contact isotopy which is generated by a contact Hamiltonian $H_t$ with $H_t|_{L_t}$ positive (negative). Note that if a contact isotopy $\phi_t$ of $M$ is positive (negative) then the Legendrian isotopy $\{\text{gr}(\phi_t)\}$ in $M\times M\times\mathbb{R}$ is negative (positive). Indeed, given a contactomorphism $\phi$ of $M$ with $\phi^{\ast}\alpha=e^g\alpha$ we can consider the induced contactomorphism $\overline{\phi}$ of $M\times M\times\mathbb{R}$ defined by $\overline{\phi}(q_1,q_2,\theta) = \big(q_1,\phi(q_2), \theta+g(q_2)\big)$. Then $\text{gr}(\phi) = \overline{\phi}(\Delta)$ where $\Delta = \{\,\big(q,q,0\big)\,|\,q\in M\,\}$. If $\{\phi_t\}$ is a positive (negative) contact isotopy then the induced contact isotopy $\{\overline{\phi_t}\}$ of $M\times M\times\mathbb{R}$ is negative (positive). Hence $\{\text{gr}(\phi_t)\} = \{\overline{\phi_t}(\Delta)\}$ is a negative (positive) Legendrian isotopy.

\begin{proof}[Proof of Lemma \ref{lemma: decomposition}]
Note first that for any other contact isotopy $\{\varphi_t\}$ the initial $\{\phi_t\}$ coincides with $\{\varphi_t^{\phantom{t}-1}\circ(\varphi_t\circ\phi_t)\}$, and so it is in the same homotopy class as the concatenation of $\{\varphi_t\circ\phi_t\}$ and $\{\varphi_t^{\phantom{t}-1}\circ(\varphi_1\circ\phi_1)\}$. If we take $\{\varphi_t\}$ to be positive then $\{\varphi_t^{\phantom{t}-1}\circ(\varphi_1\circ\phi_1)\}$ is negative and, if the Hamiltonian generating $\{\varphi_t\}$ is sufficiently big, $\{\varphi_t\circ\phi_t\}$ is positive. Thus it is enough to show that if $\{\phi_t\}_{t\in[0,1]}$ is a positive contact isotopy which is sufficiently $\mathcal{C}^1$-small then it is embedded (the case of a negative contact isotopy is of course analogous). By Weinstein's theorem, a small neighborhood of the diagonal $\Delta$ in $M\times M \times \mathbb{R}$ is contactomorphic to a neighborhood of the $0$-section of $J^1\Delta$. Since all the $\phi_t$ are $\mathcal{C}^1$-small their graphs belong to this small neighborhood of $\Delta$, and moreover they correspond to sections of $J^1\Delta$. But all Legendrian sections of $J^1\Delta$ are 1-jets of functions, so we see that the contact isotopy $\{\phi_t\}$ of $M$ corresponds to a Legendrian isotopy of $J^1\Delta$ of the form $\{j^1f_t\}$ for a certain family of functions $f_t$ on $\Delta$. As discussed above, since $\{\phi_t\}$ is positive we have that $\{j^1f_t\}$ is a negative Legendrian isotopy. We will now show that this implies that the family $f_t$ is (strictly) decreasing, and therefore that $\bigcup_t j^1f_t$ is embedded. Let $\{\Psi_{\phi_t}\}$ be a contact isotopy of $J^1\Delta$ that locally coincides with the contact isotopy $\{\overline{\phi_t}\}$ of $M\times M\times\mathbb{R}$ induced by $\{\phi_t\}$. Then we have that $\Psi_{\phi_t}(j^10) = j^1f_t$ for all $t$. Denote by $H_t: J^1\Delta\rightarrow\mathbb{R}$ the Hamiltonian of $\{\Psi_{\phi_t}\}$. Note that $H_t\big(j^1(f_t(x)\big) < 0$ for all $x\in\Delta$, because $\{\Psi_{\phi_t}\}$ is negative since $\{\phi_t\}$ is positive. Moreover, by the Hamilton-Jacobi equation (see \cite[Section 46]{A} or \cite{Ch}) we have
$$H_t\big(j^1(f_t(x)\big) = \frac{d}{dt} f_t(x).$$
We see thus that the family $f_t$ is decreasing. But this implies that $\bigcup_t j^1f_t$ is embedded and hence that the contact isotopy $\{\phi_t\}$ is embedded, as we wanted.
\end{proof}

Consider the subset $\widetilde{\mathcal{E}}$ of $\widetilde{\text{Cont}_0}(M,\xi)$ defined by
$$
\widetilde{\mathcal{E}} = \{\,[\{\phi_t\}_{t\in[0,1]}] \in \widetilde{\text{Cont}_0}(M,\xi)\setminus [\text{id}] \:|\: \bigcup_{t\in[0,1]}\text{gr}(\phi_t) \text{ is embedded}\,\}
$$
Lemma \ref{lemma: decomposition} proves that $\widetilde{\mathcal{E}}$ is a generating set for $\widetilde{\text{Cont}_0}(M,\xi)$.

Recall that a subset $S$ of a group $G$ is called a \textit{generating set} if $G = \bigcup_{k\geq 0} S^k$ where $S^0 = \{1\}$ and $S^k = S \cdot S^{k-1}$. If $S^{-1} = S$ then we obtain a (non-degenerate) norm on $G$ (called the \textit{word norm} with respect to the generating set $S$) by defining 
$$
\nu_S (g) := \text{min} \{\,k\geq 0\:|\: g \in S^k\,\}.
$$
Note that if $S$ is invariant by conjugation then $\nu_S$ is a conjugation-invariant norm on $G$. Recall that a function $\nu: G \rightarrow [0,\infty)$ is called a \emph{non-degenerate conjugation-invariant norm} if it satisfies the following properties:
\begin{enumerate}
\renewcommand{\labelenumi}{(\roman{enumi})}
\item (positivity) $\nu(g)\geq 0$ for all $g\in G$.
\item (non-degeneracy) $\nu(g)=0$ if and only if $g=\text{id}$.
\item (symmetry) $\nu(g)=\nu(g^{-1})$.
\item (triangle inequality) $\nu(fg)\leq \nu(f) + \nu(g)$.
\item (conjugation-invariance) $\nu(f) = \nu(gfg^{-1})$.
\end{enumerate}

Recall also that any conjugation-invariant norm $\nu$ on $G$ induces a bi-invariant metric $d_{\nu}$ on $G$, by defining $d(f,g) = \nu (f^{-1}g)$. We refer to \cite{Gal-Kedra} and the references therein for more background on bi-invariant word metrics.

\begin{defi}\label{definition: discriminant norm}
The \textbf{discriminant metric} on $\widetilde{\text{Cont}_0}(M,\xi)$ is the word metric with respect to the generating set $\widetilde{\mathcal{E}}$.
\end{defi}

In view of the above general discussion, in order to prove that this definition gives rise to a (non-degenerate) bi-invariant metric on $\widetilde{\text{Cont}_0}(M,\xi)$ we need to prove that the generating set $\widetilde{\mathcal{E}}$ is invariant by conjugation and satisfies $\widetilde{\mathcal{E}}^{-1} = \widetilde{\mathcal{E}}$. These properties can be seen as follows. Note that a contact isotopy $\{\phi_t\}_{t\in[0,1]}$ fails to be embedded if and only if there are two values $t_0$ and $t_1$ for which $\text{gr}(t_0)$ and $\text{gr}(t_1)$ intersect. Because of the following argument, this is equivalent to saying that $\text{gr}(\phi_{t_0}^{\phantom{t}-1}\circ\phi_{t_1})$ intersects the diagonal $\Delta$. Given a contactomorphism $\phi$ of $M$, consider the induced contactomorphism $\overline{\phi}$ of $M \times M \times \mathbb{R}$ (as defined above). Note that for any two contactomorphisms $\phi$ and $\psi$ we have $\overline{\psi\circ\phi} = \overline{\psi} \circ \overline{\phi}$. We see thus that $\text{gr}(\phi_{t_0}) = \overline{\phi_{t_0}}(\Delta)$ and $\text{gr}(\phi_{t_1}) = \overline{\phi_{t_1}}(\Delta)$ intersect (and thus $\{\phi_t\}$ fails to be embedded) if and only if $\overline{\phi_{t_0}^{\phantom{t}-1}\circ\phi_{t_1}}(\Delta) = \text{gr}(\phi_{t_0}^{\phantom{t}-1}\circ\phi_{t_1})$ intersects the diagonal $\Delta$. Equivalently, this also shows that $\{\phi_t\}_{t\in[0,1]}$ fails to be embedded if and only if there are two values $t_0$ and $t_1$ for which $\phi_{t_0}^{\phantom{t}-1}\circ\phi_{t_1}$ has a discriminant point. Indeed, recall that a point $q$ of $M$ is called a \emph{discriminant point} of a contactomorphism $\phi$ if $\phi(q)=q$ and $g(q)=0$ where $g:M\rightarrow\mathbb{R}$ is the function satisfying $\phi^{\ast}\alpha=e^g\alpha$. Clearly, $q$ is a discriminant point for $\phi$ if and only if $(q,q,0)$ belongs to the intersection of the graph of $\phi$ and the diagonal $\Delta$. In view of this discussion, the property $\widetilde{\mathcal{E}}^{-1} = \widetilde{\mathcal{E}}$ follows immediately from the fact that we have a 1-1 correspondence between discriminant points of a contactomorphism $\phi$ and discriminant points of $\phi^{-1}$: indeed, a point $q$ is a discriminant point of $\phi$ if and only if $\phi(q)$ is a discriminant point for $\phi^{-1}$. Similarly, the conjugation-invariance property of $\widetilde{\mathcal{E}}$ follows from the fact that we have a 1-1 correspondence between discriminant points of a contactomorphism $\phi$ and discriminant points of a conjugation $\psi\phi\psi^{-1}$: indeed, a point $q$ is a discriminant point of $\phi$ if and only if $\psi(q)$ is a discriminant point for $\psi\phi\psi^{-1}$.

Note that the definition of the discriminant length is independent of the choice of the contact form $\alpha$. Indeed, a point $q$ of $M$ is a discriminant point of $\phi$ with respect to the contact form $\alpha$ if and only if it is a discriminant point of $\phi$ with respect of any other contact form $\alpha'$ for $\xi$. To see this, write $\alpha'=e^h\alpha$ and $\phi^{\ast}\alpha=e^g\alpha$. Then $\phi^{\ast}\alpha'=e^{h\circ\phi+ g-h}\alpha'$. Since $q$ is a discriminant point of $\phi$ with respect to $\alpha$ we have that $\phi(q)=q$ and $g(q)=0$. But then we also have $(h\circ\phi+g-h)(q)=0$ and so $q$ is also a discriminant point of $\phi$ with respect to $\alpha'$.

In Sections \ref{section: boundedness}, \ref{section: RxS} and \ref{section: projective space} we will study in some special cases the problem of understanding for which contact manifolds the discriminant metric is unbounded, and hence not equivalent to the trivial metric. In Section \ref{section: boundedness} we will show that the discriminant norm is bounded for Euclidean space $\mathbb{R}^{2n+1}$ and for the sphere $S^{2n+1}$. In Sections \ref{section: RxS} and \ref{section: projective space} respectively we will then use generating functions to prove that the discriminant norm is unbounded for $\mathbb{R}^{2n}\times S^1$ and for projective space $\mathbb{R}P^{2n+1}$. It would be interesting to try to apply the technology of $J$-holomorphic curves or the new microlocal theory of sheaves to study this problem for more general contact manifolds. It is not clear to us for which class of contact manifolds we should expect our metric to be unbounded. As the example of the sphere shows, the presence of a 1-periodic Reeb flow is not sufficient. On the other hand, as far as we understand, this condition might not even be necessary. Note that an important difference between the cases of $\mathbb{R}^{2n}\times S^1$ and $S^{2n+1}$, that could be behind the different behavior of the discriminant metric on these two manifolds, is the following. Although on both $\mathbb{R}^{2n}\times S^1$ and $S^{2n+1}$ we have a 1-periodic Reeb flow, in the case of $S^{2n+1}$ the Reeb flow is trivial in homology (all closed Reeb orbits are contractible) while in the case of $\mathbb{R}^{2n}\times S^1$ the Reeb flow generates the fundamental group of the manifold. In view of these examples it seems reasonable to expect that the discriminant metric should be unbounded if and only if there exists a closed non-contractible Reeb orbit with respect to some contact form.

\begin{q} Let $\alpha$ be a contact form for $(M,\xi )$ and consider the Reeb flow $\varphi_t$. If $T_0$ is the minimal period of a closed Reeb orbit then, for all $t_0\in \R$, the isotopy $\{\varphi_t\}_{t\in [t_0,t_0 +T_0)}$ is embedded. Thus, if $T$ is not a multiple of $T_0$, $\{\varphi_t\}_{t\in [0,T]}$ has discriminant length smaller or equal than $\lceil \frac{T}{T_0} \rceil$ (where $\lceil\cdot\rceil$ denotes the smallest integer which is greater or equal than a given number). Is there a condition on $(M,\xi )$ (orderability? non-existence of closed contractible Reeb orbits?...) which ensures that the length is exactly $\lceil \frac{T}{T_0} \rceil$, i.e. that $\{ \varphi_t \}_{t\in [0,T]}$ is a geodesic?
\end{q}

Another interesting question is the following. Assume that $(M,\xi)$ is orderable, i.e. that the natural relation $\leq_{EP}$ induced by positive contact isotopies is a partial order on the universal cover of the contactomorphism group \cite{EP}. In this case, is the discriminant metric compatible with the partial order $\leq_{EP}$? In other words, does $[\{\phi_t\}] \leq_{EP} [\{\psi_t\}] \leq_{EP} [\{\varphi_t\}]$ imply that $d([\{\phi_t\}],[\{\psi_t\}])\leq d([\{\phi_t\}],[\{\varphi_t\}])$? Although we cannot answer this question directly for the discriminant metric, in the next section we construct a variation of the discriminant metric, that we call the oscillation pseudo-metric, which can be easily seen to be compatible with the partial order of Eliashberg and Polterovich. We believe that the oscillation pseudo-metric should also help understanding the relation between our construction and those of the second author \cite{mio2}, Zapolsky \cite{Z}, Fraser, Polterovich and Rosen \cite{FPR} and Albers and Merry \cite{AM}.

\section{The oscillation pseudo-metric, and the relation with the partial order of Eliashberg and Polterovich}\label{section: oscillation}

We will first describe a slight variation of the discriminant metric, which we call the \textit{zigzag metric}. We will then modify the construction a bit more to obtain the oscillation pseudo-metric, which we will show to be  non-degenerate if and only if the contact manifold is orderable and in this case to be compatible with the partial order.

From the proof of Lemma \ref{lemma: decomposition} it is clear that we can represent a contact isotopy $\{\phi_t\}_{t\in [0,1]}$ of $M$ as described in the statement of Lemma \ref{lemma: decomposition} and moreover requiring that every piece  $\{\phi_t\}_{t\in[t_i,t_{i+1}]}$ be either positive or negative (not necessarily with alternating signs). Therefore the set $\widetilde{\mathcal{E}}_{\pm} = \widetilde{\mathcal{E}}_+ \cup \widetilde{\mathcal{E}}_-$ where 
$$
\widetilde{\mathcal{E}}_+ = \{\,[\{\phi_t\}_{t\in[0,1]}] \in \widetilde{\text{Cont}_0}(M,\xi) \:|\: \bigcup_{t\in[0,1]}\text{gr}(\phi_t) \text{ is embedded and }\{\phi_t\} \text{ is positive}\,\}
$$
and
$$
\widetilde{\mathcal{E}}_- = \{\,[\{\phi_t\}_{t\in[0,1]}] \in \widetilde{\text{Cont}_0}(M,\xi) \:|\: \bigcup_{t\in[0,1]}\text{gr}(\phi_t) \text{ is embedded and }\{\phi_t\} \text{ is negative}\,\}
$$
is a generating set for $\widetilde{\text{Cont}_0}(M,\xi)$. Note that $\widetilde{\mathcal{E}}_{\pm}$ is invariant by conjugation, and that $\widetilde{\mathcal{E}}_{\pm}^{\phantom{\pm}-1} = \widetilde{\mathcal{E}}_{\pm}$. We define the \emph{zigzag metric} on $\widetilde{\text{Cont}_0}(M,\xi)$ to be the word metric with respect to the generating set $\widetilde{\mathcal{E}}_{\pm}$. Note that the zigzag metric is always greater than or equal to the discriminant metric. We suspect that the discriminant and the zigzag metrics should be equivalent, in the sense of \cite{BIP}\footnote{In Section \ref{section: Legendrian} we will prove this for the Legendrian discriminant and zigzag lengths. However, it is not clear to us how the argument should be modified in the context of contactomorphisms.}.

A more interesting variation of the discriminant metric is given by the following definition.

\begin{defi}
Given an element $[\{\phi_t\}_{t\in[0,1]}]$ of $\widetilde{\text{Cont}_0}(M,\xi)$, we define $\nu^+\big([\{\phi_t\}]\big)$ to be the minimal number of positive pieces in a decomposition of $[\{\phi_t\}_{t\in[0,1]}]$ as a product of elements in $\widetilde{\mathcal{E}}_{\pm}$, and $\nu^-\big([\{\phi_t\}]\big)$ to be minus\footnote{We put a minus sign here in order to be consistent with the notation for the metrics in \cite{mio2}, \cite{Z} and \cite{FPR}.} the minimal number of negative pieces in such a decomposition. We then define the \textbf{oscillation pseudo-norm}\footnote{As in \cite{Z}, we can also define a second norm $\nu'_{\text{osc}}$ by posing $\nu'_{\text{osc}} \big([\{\phi_t\}]\big) = \text{max} \big(\nu^+\big([\{\phi_t\}]\big),-\nu^-\big([\{\phi_t\}]\big)\big)$. Note that $\nu_{\text{osc}}$ and $\nu'_{\text{osc}}$ are equivalent since for every $[\{\phi_t\}_{t\in[0,1]}]$ we have that $\nu'_{\text{osc}} \big([\{\phi_t\}]\big) \leq \nu_{\text{osc}} \big([\{\phi_t\}]\big) \leq 2\nu'_{\text{osc}} \big([\{\phi_t\}]\big)$.} of $[\{\phi_t\}_{t\in[0,1]}]$ by
$$
\nu_{\text{osc}} \big([\{\phi_t\}]\big) = \nu^+\big([\{\phi_t\}]\big) - \nu^-\big([\{\phi_t\}]\big).
$$
\end{defi}

Note that the oscillation pseudo-norm of $[\{\phi_t\}_{t\in[0,1]}]$ is always smaller to or equal than its zigzag norm, but in general it is not the same. For example if $\{\phi_t\}_{t\in[0,1]}$ is a positive contractible loop of contactomorphisms then $[\{\phi_t\}_{t\in[0,\frac{1}{2}]}]$ has oscillation pseudo-norm equal to zero, while its zigzag norm is at least $1$.

\begin{prop}\label{proposition: oscillation}
The oscillation pseudo-norm $\nu_{\text{osc}}$ is a conjugation-invariant pseudo-norm on $\widetilde{\text{Cont}_0}(M,\xi)$. It is non-degenerate if and only if $(M,\xi)$ is orderable.
\end{prop}

Before proving the proposition we recall in some more details the notion of orderability. For any contact manifold $(M,\xi)$ we can consider the relation $\leq_{EP}$ on $\widetilde{\text{Cont}_0}(M,\xi)$ given by setting $$[\{\phi_t\}_{t\in[0,1]}] \leq_{EP} [\{\psi_t\}_{t\in[0,1]}]$$ if $[\{\psi_t\}_{t\in[0,1]}] \cdot [\{\phi_t\}_{t\in[0,1]}]^{-1}$ can be represented by a non-negative contact isotopy. This relation is always reflexive and transitive. If it also antisymmetric then it defines a partial order on $\widetilde{\text{Cont}_0}(M,\xi)$, and the contact manifold $(M,\xi)$ is said to be orderable. As was proved by Eliashberg, Kim and Polterovich \cite{EKP}, $S^{2n+1}$ is not orderable. On the other hand, $\mathbb{R}P^{2n+1}$, $\mathbb{R}^{2n+1}$ and $\mathbb{R}^{2n} \times S^1$ are orderable. This was proved respectively by Eliashberg and Polterovich \cite{EP} for $\mathbb{R}P^{2n+1}$, using Givental's non-linear Maslov index \cite{Giv}, and by Bhupal \cite{B} for $\mathbb{R}^{2n+1}$ and $\mathbb{R}^{2n} \times S^1$ (see also \cite{mio1}).

\begin{proof}[Proof of Proposition \ref{proposition: oscillation}]
The result is implied by the following properties of the numbers $\nu^+$ and $\nu^-$:
\begin{enumerate}
\renewcommand{\labelenumi}{(\roman{enumi})}
\item For any $[\{\phi_t\}_{t\in[0,1]}]$ we have $\nu^+\big([\{\phi_t\}]\big) \geq 0$ and $\nu^-\big([\{\phi_t\}]\big)\leq 0$.
\item $\nu^-(\{\phi_t\}) = - \,\nu^+(\{\phi_t^{-1}\})$.
\item $\nu^+(\{\psi_t\phi_t\}) \leq \nu^+(\{\phi_t\}) + \nu^+(\{\psi_1\phi_t\})$ and $\nu^-(\{\psi_t\phi_t\}) \geq \nu^-(\{\phi_t\}) + \nu^-(\{\psi_1\phi_t\})$.
\item $\nu^{\pm}(\{\phi_t\}) = \nu^{\pm}(\{\psi_t\phi_t\psi_t^{-1}\})$.
\item If $(M,\xi)$ is orderable then $\nu^+(\{\phi_t\})=\nu^-(\{\phi_t\})= 0$ if and only if $\{\phi_t\}$ is the identity.
\end{enumerate}
Note that (v) follows from the criterion proved by Eliashberg and Polterovich \cite{EP}: $(M,\xi)$ is orderable if and only if there is no positive contractible loop of contactomorphisms. All other properties are immediate.
\end{proof}

\begin{q}
Assume that $(M,\xi)$ is orderable. Then, are there examples of contact isotopies for which the oscillation norm is strictly smaller than the zigzag norm\footnote{See \cite[page 23]{P} for a similar open question in Hofer geometry.}? If yes, is the oscillation norm at least always equivalent to the discriminant and zigzag norms?
\end{q}

We now assume that $(M,\xi)$ is orderable, and we denote by $d_{\text{osc}}$ the non-degenerate bi-invariant metric on $\widetilde{\text{Cont}_0}(M,\xi)$ which is induced by the oscillation norm. As we will now see, $d_{\text{osc}}$ is compatible with the partial order $\leq_{EP}$ introduced by  Eliashberg and Polterovich, i.e.
$$\big(\,\widetilde{\text{Cont}_0}(M,\xi),\,d_{\text{osc}}\,,\leq_{EP}\big)$$
is a partially ordered metric space. Recall that a \emph{partially ordered metric space} is a metric space $(Z,d)$ endowed with a partial order $\leq$ such that for every $a$, $b$, $c$ in $Z$ with $a\leq b\leq c$ it holds $d(a,b)\leq d(a,c)$.

\begin{prop}
For every orderable contact manifold $(M,\xi)$, the bi-invariant metric $d_{\text{osc}}$ on $\widetilde{\text{Cont}_0}(M,\xi)$ is compatible with the partial order $\leq_{EP}$.
\end{prop}

\begin{proof}
It is enough to show that if $\text{id} \leq [\{\phi_t\}] \leq [\{\psi_t\}]$ then $\nu_{\text{osc}} \big([\{\phi_t\}]\big) \leq \nu_{\text{osc}} \big([\{\psi_t\}]\big)$. Since $[\{\phi_t\}] \geq \text{id}$, $[\{\phi_t\}]$ can be represented by a non-negative contact isotopy and so $\nu^-\big([\{\phi_t\}]\big) = 0$ i.e. $\nu_{\text{osc}} \big([\{\phi_t\}]\big) = \nu^+\big([\{\phi_t\}]\big)$. Similarly, $\nu_{\text{osc}} \big([\{\psi_t\}]\big) = \nu^+\big([\{\psi_t\}]\big)$. If we suppose that $\nu_{\text{osc}} \big([\{\psi_t\}]\big) < \nu_{\text{osc}} \big([\{\phi_t\}]\big)$ then this means that $[\{\psi_t\}]$ can be represented by a contact isotopy such that the number of positive embedded pieces is less than $\nu^+\big([\{\phi_t\}]\big)$. But this gives a contradiction because $[\{\phi_t\}]$ is also represented by the concatenation of $\{\psi_t\}$ and the contact isotopy
$\{\phi_t \circ \psi_t^{\phantom{t}-1} \circ \psi_1\}$, which is negative since by hypothesis $[\{\phi_t\}] \leq [\{\psi_t\}]$.
\end{proof}

Note that the definition of the oscillation pseudo-norm is a special case of a general construction in the context of partially ordered groups and word metrics. Let $G$ be a group. A subset $C$ of $G$ is called a \textit{cone} if $1 \in C$ and $C^2 \subset C$. Every cone $C$ induces a relation $\leq_C$ on $G$ defined by
$$
a\leq_C b \text{ if }a^{-1}b \in C.
$$
This relation is always transitive and reflexive, and it is bi-invariant if $C$ is invariant by conjugation. If it is also anti-symmetric (i.e. $a\leq_C b$ and $b \leq_C a \Rightarrow a=b$) then it is a bi-invariant partial order on $G$. Assume now that $C$ has a generating set $S$ such that $S \cup S^{-1}$ is a generating set for $G$. We can then define the \textit{counting functions} $\nu_S^{\phantom{S}\pm}: G \rightarrow \mathbb{Z}_{\geq 0}$ by setting $\nu_S^{\phantom{S}+}(g)$ to be the minimal number of elements of $S$ in a representation of $g$ as a product of elements of $S$ and $S^{-1}$, and $\nu_S^{\phantom{S}-}(g)$ to be minus the minimal number of elements of $S^{-1}$ in such a representation. Then $\nu_S:=\nu_S^{\phantom{S}+} - \nu_S^{\phantom{S}-}$ is a pseudo-metric on $G$, bi-invariant if $S$ is invariant by conjugation. Moreover we have that $\nu_S$ is non-degenerate if and only if $\leq_C$ is a partial order, and in this case $(G,\nu_S, \leq_C)$ is a partially ordered metric space. The definition of the oscillation metric falls into this general construction, by considering the cone formed by elements of $\widetilde{\text{Cont}_0}(M,\xi)$ that can be represented by a non-negative contact isotopy, with generating set $\widetilde{\mathcal{E}}_+$.

\begin{rmk}
As mentioned above, other integer-valued bi-invariant metrics have been defined by the second author \cite{mio1} for $\mathbb{R}^{2n} \times S^1$, by Zapolsky \cite{Z} for $T^{\ast}B \times S^1$ (for $B$ closed) and by Fraser, Polterovich and Rosen \cite{FPR} and Albers-Merry \cite{AM} for some other classes of circle bundles. For these manifolds we expect that the oscillation metric should be equivalent to these metrics (see also Remark \ref{remark: equivalence with my metric} for the case of $\mathbb{R}^{2n} \times S^1$).
\end{rmk}

\begin{rmk}
Let $(M,\xi)$ be an orderable contact manifold. Given a domain $\mathcal{U}$ in $(M,\xi)$ we can define its \emph{capacity} by
$$c(\mathcal{U})=\text{sup}\,\{\,\nu^+([\{\phi_t\}])\:|\: \phi_t \text{ supported in }\mathcal{U}\,\}$$
and its \emph{displacement energy} by
$$ E(\mathcal{U}):=\text{inf}\;\{\; \nu_{\text{osc}}([\{\psi_t\}]) \;|\; \psi_1(\mathcal{U})\cap\mathcal{U}= \emptyset \;\}.$$
Since the $\nu^+$ and $\nu^-$ are invariant by conjugation, the capacity and displacement energy of a domain are contact invariants. In the case of $M = \mathbb{R}^{2n} \times S^1$ we think that these invariants should reduce to the capacity and displacement energy that were introduced in \cite{mio1} to prove the contact non-squeezing theorem. It would be interesting to study the capacity and displacement energy with respect to the oscillation metric for domains in more general contact manifolds.
\end{rmk}

\section{Boundedness for $\mathbb{R}^{2n+1}$ and $S^{2n+1}$}\label{section: boundedness}

We first show that the discriminant, zigzag and oscillation metrics are bounded for the universal cover of the identity component of the group of compactly supported contactomorphisms of  $\mathbb{R}^{2n+1}$, with its standard contact structure $\xi=\text{ker}\big(dz+\frac{1}{2}(xdy-ydx)\big)$. Since the discriminant and oscillation metrics are bounded above by the zigzag metric, it is enough to show that the zigzag metric is bounded. This is proved in the following proposition.

\begin{prop}\label{proposition: euclidean space}
The zigzag norm of the homotopy class of a compactly supported contact isotopy $\{\phi_t\}_{t\in[0,1]}$ of $\mathbb{R}^{2n+1}$ is always smaller than or equal to $2$.
\end{prop}

\begin{proof}
As in the proof of Lemma \ref{lemma: decomposition}, $\{\phi_t\}$ is in the same homotopy class as the concatenation of $\{\varphi_t\circ\phi_t\}$ and $\{\varphi_t^{\phantom{t}-1}\circ(\varphi_1\circ\phi_1)\}$ for any other contact isotopy $\{\varphi_t\}$. We will show that it is possible to choose $\{\varphi_t\}$ in such a way that $\{\varphi_t\circ\phi_t\}$ is positive, $\{\varphi_t^{\phantom{t}-1}\circ(\varphi_1\circ\phi_1)\}$ is negative and both are embedded. Note first that for a contact Hamiltonian $H: \mathbb{R}^{2n+1} \rightarrow \mathbb{R}$ of the form $H(x,y,z) = H(x,y)$ the corresponding contact vector field is given by 
$$
X_H = - \frac{\partial H}{\partial y}\,\frac{\partial}{\partial x} + \frac{\partial H}{\partial x}\,\frac{\partial}{\partial y} + \big(H(x,y) - \frac{1}{2}(x\frac{\partial H}{\partial x} + y\frac{\partial H}{\partial y})\big)\,\frac{\partial}{\partial z}\;.
$$
Consider now $H_{\rho,R}(x,y,z) = \rho(\frac{x^2+y^2}{R})$ for some positive real number $R$ and some function $\rho: [0,\infty) \rightarrow [0,\infty)$ which is supported in $[0,1]$ and has $\rho'<0$ and $\rho''>0$ (Hamiltonian functions of this form have also been used by Traynor \cite{Tr} in her calculations of symplectic homology of ellipsoids). The corresponding vector field is given by 
$$
X_{\rho,R} = - \frac{2y}{R}\rho'(\frac{x^2+y^2}{R})\,\frac{\partial}{\partial x} + \frac{2x}{R}\rho'(\frac{x^2+y^2}{R})\,\frac{\partial}{\partial y} + \big(\rho(\frac{x^2+y^2}{R}) - \frac{x^2+y^2}{R} \rho'(\frac{x^2+y^2}{R})\big)\,\frac{\partial}{\partial z}\;.
$$
thus the generated contact isotopy $\{\varphi_t^{\phantom{t}\rho,R}\}$ is embedded, since the coefficient of $\frac{\partial}{\partial z}$ is always strictly positive. We get thus that  $\{(\varphi_t^{\phantom{t}\rho,R})^{-1}\circ(\varphi_1\circ\phi_1)\}$ is embedded (and negative) for any choice of $\rho$ and $R$. Moreover if we choose $\rho$ and $R$ big enough then $\{\varphi_t^{\phantom{t}\rho,R}\circ\phi_t\}$ will also be embedded (and positive). Indeed, our initial contact isotopy $\{\phi_t\}$ is compactly supported and so in particular every point is moved in the $z$-direction only by a bounded quantity. Thus, by choosing $\rho$ and $R$ big enough we can make $\{\varphi_t^{\phantom{t}\rho,R}\}$ move every point of the support of $\{\phi_t\}$ far enough in the $z$-direction so that $\{\varphi_t^{\phantom{t}\rho,R}\circ\phi_t\}$ is embedded. This does not finish the proof yet, because $\{\varphi_t^{\phantom{t}\rho,R}\}$ is not a compactly supported contact isotopy (it is not compactly supported in the $z$-direction). In order to fix this we cut-off the Hamiltonian $H_{\rho,R}$ in the $z$-direction, far away from the support of $\{\phi_t\}$. More precisely, we consider the contact Hamiltonian $H_{\rho,R,f}(x,y,z) = f(z)\big(\rho(\frac{x^2+y^2}{R})\big)$ for some positive cut-off function $f$. Then 
\begin{eqnarray*}
X_{\rho,R,f} &=& \Big(-\frac{2y}{R}\rho'(\frac{x^2+y^2}{R}) + \frac{x}{2}f'(z)\rho(\frac{x^2+y^2}{R})\Big)\,\frac{\partial}{\partial x}\\
&+& \Big(\frac{2x}{R}\rho'(\frac{x^2+y^2}{R}) + \frac{y}{2}f'(z)\rho(\frac{x^2+y^2}{R})\big)\Big)\,\frac{\partial}{\partial y}\\
&+& f(z)\Big(\big(\rho(\frac{x^2+y^2}{R}) - \frac{x^2+y^2}{R} \rho'(\frac{x^2+y^2}{R})\big)\Big)\,\frac{\partial}{\partial z}\;.
\end{eqnarray*}
By the same arguments as before, if we choose $\rho$ and $R$ big enough then for the generated contact isotopy $\{\varphi_t\} := \{\varphi_t^{\phantom{t}\rho,R,f}\}$ we have that $\{\varphi_t\circ\phi_t\}$ is positive, $\{\varphi_t^{\phantom{t}-1}\circ(\varphi_1\circ\phi_1)\}$ is negative and both are embedded.
\end{proof}

As a consequence of Proposition \ref{proposition: euclidean space} we see that, on any contact manifold, the discriminant, zigzag and oscillation  norms of the homotopy class of a contact isotopy that is supported in a Darboux ball are bounded.

\begin{lemma}\label{lemma: ball}
Let $(M,\xi)$ be a contact manifold and  $B\subset M$ a Darboux ball. Then the zigzag norm of the homotopy class of a contact isotopy $\{\phi_t\}_{t\in [0,1]}$ supported in $B$ is at most $2$.
\end{lemma}

\begin{proof}
This follows from Proposition \ref{proposition: euclidean space} and the fact that any Darboux ball is contactomorphic to the whole $\mathbb{R}^{2n+1}$ (see \cite{CKS}). In the notation of the proof of Proposition \ref{proposition: euclidean space}, here instead of simply considering the flow of $\lambda f \frac{\partial}{\partial z}$ we take the contact isotopy $\{\varphi_t\}$ to be the composition of the flow of $\lambda f \frac{\partial}{\partial z}$ together with the flow of $\lambda' R$, where $R$ is a fixed Reeb vector field on $M$ and $\lambda' >0$. If $\lambda'$ is small enough then $\{\varphi_t\}$ is embedded and positive. In that case, both $\{\varphi_t^{\phantom{t}-1}\}$ and $\{\varphi_t \circ \phi_t\}$ are embedded and we conclude as in Proposition \ref{proposition: euclidean space}.
\end{proof}

We will now use Lemma \ref{lemma: ball} to prove that the discriminant, zigzag and oscillation norms on $S^{2n+1}$ are bounded. Recall that the standard contact structure on $S^{2n+1}$ is defined to be the kernel of the restriction to $S^{2n+1}$ of the 1-form $\alpha=xdy-ydx$, where we regard $S^{2n+1}$ as the unit sphere in $\mathbb{R}^{2n+2}$ and where $(x,y)$ are the coordinates on $\mathbb{R}^{2n+2}$.

\begin{prop}\label{proposition: sphere}
The zigzag norm of the homotopy class of a contact isotopy $\{\phi_t\}_{t\in[0,1]}$ of $S^{2n+1}$ is always smaller than or equal to $4$.
\end{prop}

\begin{proof}
Consider a point $q$ of $S^{2n+1}$, and the arc $\gamma =\cup_{t\in [0,1]} \phi_t (q)$. Take a Darboux ball $B \subset S^{2n+1}$ that contains $\gamma$ (recall that for any point $q'$ of $S^{2n+1}$ the complement $S^{2n+1}\setminus q'$ is contactomorphic to $\mathbb{R}^{2n+1}$, see for example \cite{gei}). Consider the contact isotopy $\{v_t\}_{t\in[0,1]}$ generated by the contact Hamiltonian $\chi H_t$ where $H_t$ is the Hamiltonian of $\{\phi_t\}$ and $\chi$ a cut-off function that is $1$ on $B$ and $0$ outside a slightly bigger Darboux ball $B'\supset B$ (in order for $B'$ to exist we choose the initial $B$ such that the complement is not a single point). If $N(q)$ is a sufficiently small neighborhood of $q$ then $\phi_t\big(N(q)\big)\subset B$ for all $t\in[0,1]$. Hence $v_t \vert_{N(q)} =\phi_t \vert_{N(q)}$ and so $v_t^{\phantom{t}-1} \circ \phi_t$ is the identity on $N(q)$. Note that $B''=S^{2n+1}\setminus\overline{N(q)}$ is also a Darboux ball, and that $u_t:=v_t^{\phantom{t}-1} \circ \phi_t$ is supported in $B''$. Since $\{\phi_t=v_t\circ u_t\}$ is homotopic to the concatenation of $\{u_t\}$ and $\{v_t\circ u_1\}$ and since the zigzag length of $\{v_t\circ u_1\}$ is equal to the zigzag length of $\{v_t\}$ we see thus, by Lemma \ref{lemma: ball}, that the zigzag norm of $\{\phi_t\}$ is at most $4$.
\end{proof}

\section{Relation with the fragmentation norm}\label{section: fragmentation}

Let $M$ be a compact smooth manifold. The fragmentation lemma (see \cite{Ban}) says that every diffeomorphism $\phi$ of $M$ which is isotopic to the identity can be written as a finite product of diffeomorphisms supported in embedded open balls. The \emph{fragmentation norm} of $\phi$ is then defined to be the minimal number of factors in such a decomposition. Burago, Ivanov and Polterovich \cite{BIP} proved that if the fragmentation norm on $M$ is bounded then every conjugation-invariant norm on $\text{Diff}_0(M)$ is equivalent to the trivial one. Moreover they showed that this is the case for $M=S^n$ and for $M$ a compact connected $3$-dimensional manifold. Tsuboi \cite{Ts,Ts2} extended this result respectively to all compact connected odd-dimensional manifolds and to the even dimensional case. The fragmentation norm can be defined also in the context of symplectic topology. Let $(M,\omega)$ be a compact symplectic manifold, and $\mathcal{U}$ an open domain in $M$. The Hamiltonian fragmentation lemma \cite{Ban} says that every Hamiltonian symplectomorphism $\phi$ of $(M,\omega)$ can be written as a finite product $\phi=\phi_1\cdots\phi_N$ where each $\phi_i$ is conjugate to a Hamiltonian symplectomorphism generated by a Hamiltonian function supported in $\mathcal{U}$. The \emph{Hamiltonian fragmentation norm} of $\phi$ with respect to $\mathcal{U}$ is defined to be the minimal number of factors in such a decomposition. Using the work of Entov and Polterovich \cite{EnP} it was proved in \cite{BIP} that the Hamiltonian fragmentation norm in $\text{Ham}(\mathbb{T}^{2n})$ with respect to a displaceable domain $\mathcal{U}$ is unbounded. See also Monzner, Vichery and Zapolsky \cite{MVZ} for similar results in the cotangent bundle.

In the contact case the fragmentation lemma holds again in the more general form, as was proved by Rybicky.

\begin{lemma}[Lemma 5.2 in \cite{Ry}]
Let $(M,\xi)$ be a compact\footnote{The result is also true for compactly supported contactomorphisms of non-compact contact manifolds.} contact manifold and $\{\mathcal{U}_i\}_{i=1}^k$ an open cover. Then, every contactomorphism $\phi$ of $M$ isotopic to the identity can be written
as a finite product $\phi=\phi_1\cdots\phi_N$ where each $\phi_j$ is supported in some $\mathcal{U}_{i(j)}$. The same is true for isotopies of contactomorphisms instead of contactomorphisms.
\end{lemma}

Using this lemma, we can define the \emph{contact fragmentation norm} of a contactomorphism $\phi$ to be the minimal number of factors in a decomposition of $\phi$ as a finite product of contactomorphisms supported in a Darboux ball. The same definition can be given also for the contact fragmentation norm in the universal cover of the identity component of the contactomorphism group of $M$. As far as we know there are no results in the literature about boundedness or unboundedness of the contact fragmentation norm.

As we will now see, for any contact manifold $(M,\xi)$ the contact fragmentation norm in $\widetilde{\text{Cont}_0}(M,\xi)$ is bounded below by half of the zigzag norm. Hence, if the zigzag norm is unbounded then so is the contact fragmentation norm.

\begin{prop}\label{proposition: fragmentation}
For any contact manifold $(M,\xi)$ and any $[\{\phi_t\}]\in\widetilde{\text{Cont}_0}(M,\xi)$ we have that the zigzag norm of $[\{\phi_t\}]$ is smaller than or equal to twice its fragmentation norm.
\end{prop}

\begin{proof}
This follows from Lemma \ref{lemma: ball}.
\end{proof}

In the next two sections we will prove that the discriminant, zigzag and oscillation norms are unbounded for $M=\mathbb{R}^{2n}\times S^1$ and $M=\mathbb{R}P^{2n-1}$. Hence, by Proposition \ref{proposition: fragmentation}, the contact fragmentation norm in the universal cover of the identity component of the contactomorphism group is also unbounded in these two cases.

\section{Unboundedness for $\mathbb{R}^{2n}\times S^1$} \label{section: RxS}

We consider the standard contact structure $\xi = \text{ker} (dz-ydx)$ on $\mathbb{R}^{2n}\times S^1$. We will show in this section that the discriminant and oscillation norms on the universal cover $\widetilde{\text{Cont}^{\phantom{0}c}_0}(\mathbb{R}^{2n}\times S^1)$ of the identity component of the group of compactly supported contactomorphisms of $\mathbb{R}^{2n}\times S^1$ are unbounded. As a consequence, the zigzag norm is also unbounded.

To prove the result we will use the spectral numbers $c^{\pm}(\phi)$ associated to a compactly supported contactomorphism $\phi$ of $\mathbb{R}^{2n}\times S^1$ which is isotopic to the identity. These spectral numbers were introduced by Bhupal \cite{B} in the case of $\mathbb{R}^{2n+1}$ and by the second author \cite{mio1} for $\mathbb{R}^{2n}\times S^1$. Recall that to every $\phi$ in $\text{Cont}^{\phantom{0}c}_0(\mathbb{R}^{2n}\times S^1)$ we can associate a Legendrian submanifold $\Gamma_{\phi}$ of $J^1(S^{2n}\times S^1)$, which is defined as follows. First we will see $\phi$ as a 1-periodic contactomorphism of $\mathbb{R}^{2n+1}$, and associate to it the Legendrian submanifold $\Gamma_{\phi}$ of $J^1\mathbb{R}^{2n+1}$ which is the image of the graph of $\phi$ under the contact embedding $\tau: \mathbb{R}^{2n+1} \times \mathbb{R}^{2n+1} \times \mathbb{R}\longrightarrow J^1\mathbb{R}^{2n+1}$, $(x,y,z,X,Y,Z,\theta)\mapsto \big(x,Y,z, Y-e^{\theta}y, x-X, e^{\theta}-1, xY-XY+Z-z\big)$. We then notice that, since $\phi$ is 1-periodic, $\Gamma_{\phi}$ can be seen as a Legendrian submanifold of $J^1(\mathbb{R}^{2n}\times S^1)$. Moreover, since $\tau$ sends the diagonal to the 0-section and since $\phi$ is compactly supported, we have that $\Gamma_{\phi}$ coincides with the 0-section outside a compact set. Hence it can be seen as a Legendrian submanifold of $J^1(S^{2n}\times S^1)$, via compactification.

For a contactomorphism $\phi$ in $\text{Cont}^{\phantom{0}c}_0(\mathbb{R}^{2n}\times S^1)$ the spectral numbers $c^{\pm}(\phi)$ are then defined by $c^+(\phi)=c(\mu,\Gamma_{\phi})$ and $c^-(\phi)=c(1,\Gamma_{\phi})$ where $\mu$ and $1$ are the volume and unit classes respectively in $H^{\ast}(S^{2n}\times S^1)$. Recall that for a cohomology class $u$ in $H^{\ast}(B)$, where $B$ is a closed smooth manifold, and a Legendrian submanifold $L$ of $J^1B$, the real number $c(u,L)$ is obtained by minimax from a generating function quadratic at infinity of $L$ (see \cite{V,B,mio1}).

To prove that the discriminant and oscillation metrics are unbounded we will actually only use the spectral number $c^+$. In particular we will need the following properties (see \cite{mio1,mio2}).

\begin{lemma}\label{lemma: properties spectral invariants contactomorphisms}
\begin{enumerate}
\renewcommand{\labelenumi}{(\roman{enumi})}
\item For any contactomorphism $\phi$, there exists a translated point $q$ of $\phi$ with contact action equal to $c^+(\phi)$.
\item For any two $\phi$, $\psi$ we have $\lceil c^+(\phi\psi)\rceil\leq\lceil c^+(\phi)\rceil+\lceil c^+(\psi)\rceil$ (where $\lceil\cdot\rceil$ denotes the smallest integer greater or equal than a given real number).
\item For any contact isotopy $\{\phi_t\}$, $c^+(\phi_t)$ is continuous in $t$.
\item $c^+(\text{id}) = 0$.
\item If $\{\phi_t\}$ is a negative contact isotopy with $\phi_0=\text{id}$ then $c^+(\phi_t) = 0$ for all $t$.
\end{enumerate}
\end{lemma}

Recall that a point $q$ of $\mathbb{R}^{2n}\times S^1$ is called a \textit{translated point} of $\phi$ if $\phi(q)$ and $q$ are in the same Reeb orbit (i.e. they only differ by a translation in the $z$-direction) and $g(q)=0$ where $g$ is the function determined by $\phi^{\ast}\alpha = e^g \alpha$. If we view $\phi$ as a 1-periodic contactomorphism of $\mathbb{R}^{2n+1}$ then the \textit{contact action} of a translated point $q$ of $\phi$ is by definition the difference in the $z$-coordinate of $\phi(q)$ and $q$.

In order to prove that the discriminant and oscillation metrics on $\widetilde{\text{Cont}_0^{\phantom{0}c}}(\mathbb{R}^{2n}\times S^1)$ are unbounded, we will consider the pullback of $c^+$ to $\widetilde{\text{Cont}_0^{\phantom{0}c}}(\mathbb{R}^{2n}\times S^1)$ under the covering projection $$\widetilde{\text{Cont}_0^{\phantom{0}c}}(\mathbb{R}^{2n}\times S^1) \rightarrow \text{Cont}_0^{\phantom{0}c}(\mathbb{R}^{2n}\times S^1).$$ 
In other words, for an element $[\{\phi_t\}_{t\in[0,1]}]$ of $\widetilde{\text{Cont}_0^{\phantom{0}c}}(\mathbb{R}^{2n}\times S^1)$ we define
$\widetilde{c^+}\big([\{\phi_t\}_{t\in[0,1]}]\big) = c^+ (\phi_1)$.
Note that $\lceil \widetilde{c}^+ \rceil$ also satisfies the triangle inequality: 
$$\lceil \widetilde{c^+}\big([\{\phi_t\}] \cdot [\{\psi_t\}]\big)\rceil \leq \lceil \widetilde{c^+}\big([\{\phi_t\}])\rceil + \lceil \widetilde{c^+}\big(\{[\psi_t]\}\big)\rceil.$$
Moreover note that if $\{\phi_t\}$ is an embedded contact isotopy (with $\phi_0 = \text{id}$) then $\widetilde{c^+}\big([\{\phi_t\}]\big) \leq 1$: this follows from (i)-(iii)-(iv) of Lemma \ref{lemma: properties spectral invariants contactomorphisms}.

As discussed in \cite{mio2} there are compactly supported contactomorphisms of $\mathbb{R}^{2n}\times S^1$, isotopic to the identity, with arbitrarily big $c^+$. Thus, unboundedness of the discriminant and oscillation metrics follows if we can prove that for any element $[\{\phi_t\}]$ of $\widetilde{\text{Cont}_0^{\phantom{0}c}}(\mathbb{R}^{2n}\times S^1)$ with $\lceil \widetilde{c^+}\big([\{\phi_t\}]\big) \rceil = k > 0$ its discriminant and oscillation norms are greater than or equal to $k$. In order to prove this, suppose by contradiction that the discriminant metric of $[\{\phi_t\}]$ is $k' < k$. Then $[\{\phi_t\}] = \prod_{j=1}^{k'} [\{\phi_{j,t}\}]$ with all the $\{\phi_{j,t}\}$ embedded. By the discussion above we then have 
$$
\lceil \widetilde{c^+}\big([\{\phi_t\}]\big) \rceil \leq \sum_{j=1}^{k'} \lceil \widetilde{c^+}\big([\{\phi_{j,t}\}]\big) \rceil \leq k' < k
$$
contradicting our hypothesis. A similar argument also applies to the oscillation metric: if the oscillation norm of $[\{\phi_t\}]$ is $k' < k$ then in a decomposition of $[\{\phi_t\}]$ as the product of embedded monotone pieces ot most $k'$ of them are positive. Then as above (and additionally using property (v) of Lemma \ref{lemma: properties spectral invariants contactomorphisms}) we get that $\lceil \widetilde{c^+}\big([\{\phi_t\}]\big) \rceil \leq k' < k$.

Notice that for any contact manifold $(M,\xi)$ the discriminant metric on  $\widetilde{\text{Cont}_0^{\phantom{0}c}}(M,\xi)$ descends to a (non-degenerate) bi-invariant metric on the contactomorphism group $\text{Cont}_0^{\phantom{0}c}(M,\xi)$ by considering the word metric with respect to the generating set 
$$
\mathcal{E} = \{\,\phi \in \text{Cont}_0(M,\xi) \:|\: \exists \, [\{\phi_t\}_{t\in[0,1]}] \in \widetilde{\mathcal{E}} \text{ with }\phi_1=\phi\,\}
$$
(i.e. $\mathcal{E}$ is the image of $\widetilde{\mathcal{E}}$ under the covering projection). Similarly, the oscillation pseudo-metric on $\widetilde{\text{Cont}_0^{\phantom{0}c}}(M,\xi)$ also descends to a bi-invariant pseudo-metric on $\text{Cont}_0^{\phantom{0}c}(M,\xi)$. Note that the above proof also shows that these induced discriminant and oscillation metrics on $\text{Cont}_0^{\phantom{0}c}(\mathbb{R}^{2n}\times S^1)$ are unbounded. 

As was proved in \cite{mio1}, if $\{\phi_t\}$ is a positive contact isotopy then $t \mapsto c^+(\phi_t)$ is increasing. This, together with Lemma \ref{lemma: properties spectral invariants contactomorphisms}(iv), implies that $\mathbb{R}^{2n}\times S^1$ is orderable and so that the oscillation pseudo-metric on  $\widetilde{\text{Cont}_0^{\phantom{0}c}}(\mathbb{R}^{2n}\times S^1)$ is non-degenerate. Recall that orderability is equivalent to the non-existence of a positive contractible loop of contactomorphisms. In the case of $\mathbb{R}^{2n}\times S^1$, the monotonicity properties of $c^+$ mentioned above actually imply the stronger fact that there are no positive loops of contactomorphisms. Because of this, the partial order descends to $\text{Cont}_0^{\phantom{0}c}(\mathbb{R}^{2n}\times S^1)$, and so also the induced oscillation pseudo-metric on $\text{Cont}_0^{\phantom{0}c}(\mathbb{R}^{2n}\times S^1)$ is non-degenerate.

\begin{rmk}\label{remark: equivalence with my metric}
Recall that the metric in \cite{mio2} is defined by $d(\phi,\psi) = \lceil c^+(\psi\phi^{-1}) \rceil - \lfloor c^+(\psi\phi^{-1}) \rfloor$. It seems plausible that the oscillation metric on the contactomorphism group of $\mathbb{R}^{2n}\times S^1$ should be equal (or at least equivalent) to this metric. Notice however that the argument above only shows that the oscillation norm is greater than or equal to the norm of \cite{mio2}.
\end{rmk}

\section{Unboundedness for $\mathbb{R}P^{2n-1}$}\label{section: projective space}

Consider the real projective space $\mathbb{R}P^{2n-1}$ with the contact structure obtained by quotienting the standard contact structure on the sphere $S^{2n-1}$ by the antipodal action of $\mathbb{Z}_2$. We want to show that the discriminant, zigzag and oscillation metrics are unbounded on the universal cover of the identity component of the contactomorphism group of $\mathbb{R}P^{2n-1}$. We will consider first the discriminant metric, and prove that it is unbounded by showing that, for every $k$, the element in the universal cover generated by the $2k$-th iteration of the Reeb flow has discriminant norm at least $k$.

Recall the Reeb flow associated to the standard contact form $\alpha=xdy-ydx$ on $S^{2n-1}$ is given by the Hopf fibration $z\mapsto e^{2\pi it}z$. We denote by $\{\varphi_t\}_{t\in[0,1]}$ the $2k$-th iteration of the Reeb flow, i.e. $\varphi_t(z)=e^{4\pi i kt}z$, and we use the same notation $\{\varphi_t\}$ also for the induced contact isotopy in $\mathbb{R}P^{2n-1}$. We want to show that the discriminant length of $\{\varphi_t\}_{t\in[0,1]}$ is at least $k$. Similarly to \cite{Giv,Th2,mio5} we will prove this by studying a 1-parameter family of \textit{conical} generating functions for the lift to $\mathbb{R}^{2n}$ of a contact isotopy of $\mathbb{R}P^{2n-1}$.

We start by recalling how to lift to $\mathbb{R}^{2n}$ a contact isotopy $\{\phi_t\}_{t\in[0,1]}$ of $\mathbb{R}P^{2n-1}$. Notice first that $\{\phi_t\}_{t\in[0,1]}$ can be uniquely lifted to a $\mathbb{Z}_2$-equivariant contact isotopy of $S^{2n-1}$, that we will still denote by $\{\phi_t\}_{t\in[0,1]}$, by taking the flow of the pullback of the contact Hamiltonian under the projection $S^{2n-1}\rightarrow\mathbb{R}P^{2n-1}$. For every $t\in[0,1]$ we then define the lift $\Phi_t: \mathbb{R}^{2n} \rightarrow\mathbb{R}^{2n}$ of $\phi_t: S^{2n-1}\rightarrow S^{2n-1}$ by the formula
\begin{equation}\label{equation: lift}
\Phi_t(z)= \frac{|z|}{ e^{ \frac{1}{2}\,g_t( \frac{z}{|z|} ) }  }\,\phi_t(\frac{z}{|z|})
\end{equation}
where we identify $\mathbb{R}^{2n}$ with $\mathbb{C}^n$, and where $g_t:S^{2n-1}\rightarrow S^{2n-1}$ is the function determined by $\phi_t^{\phantom{t}\ast}\alpha=e^{g_t}\alpha$. Although $\Phi_t$ is only defined on $\mathbb{R}^{2n}\setminus 0$, we extend it continuously to the whole $\mathbb{R}^{2n}$ by posing $\Phi_t(0)=0$. Recall that, more generally, every contactomorphism $\phi$ of a contact manifold $\big(M,\xi=\text{ker}(\alpha)\big)$ can be lifted to a $\mathbb{R}$-equivariant symplectomorphism $\Phi$ of the symplectization $\big(SM=M\times\mathbb{R}\,,\, \omega=d(e^{\theta}\alpha)\big)$ by defining $\Phi(q,\theta)=\big(\phi(q),\theta-g(q)\big)$. If we identify $S(S^{2n-1})=S^{2n-1}\times\mathbb{R}$ with $\mathbb{R}^{2n}\setminus 0$ by the symplectomorphism $(q,\theta)\mapsto \sqrt{2}\, e^{\frac{\theta}{2}}q$ then this formula for the lift $\Phi$ reduces to (\ref{equation: lift}). If $\{\phi_t\}_{t\in[0,1]}$ is a contact isotopy of $\big(M,\xi=\text{ker}(\alpha)\big)$ generated by the contact Hamiltonian $h_t: M \rightarrow \mathbb{R}$, then the lift $\{\Phi_t\}_{t\in[0,1]}$ is the Hamiltonian isotopy of $SM$ which is generated by the $\mathbb{R}$-equivariant Hamiltonian $H_t: SM \rightarrow \mathbb{R}$, $H_t(q,\theta) = e^{\theta}h_t(q)$. In the case of $M=S^{2n-1}$, if $\{\phi_t\}_{t\in[0,1]}$ is generated by $h_t: S^{2n-1}\rightarrow\mathbb{R}$ then $\Phi_t: \mathbb{R}^{2n} \rightarrow\mathbb{R}^{2n}$ is generated by $H_t: \mathbb{R}^{2n} \rightarrow\mathbb{R}$, $H_t(z) = \frac{|z|^2}{2} h_t\big(\frac{z}{|z|}\big)$. Note that $\mathbb{R}$-equivariance of the Hamiltonian $H_t$ reduces in the case of the sphere to the property of being homogeneous of degree $2$, i.e. $H_t(\lambda z) = \lambda^2H_t(z)$ for every $\lambda\in\mathbb{R}_+$. Moreover, if $h_t: S^{2n-1}\rightarrow\mathbb{R}$ is the lift of a function on $\mathbb{R}P^{2n-1}$ then $H_t$ is \textit{conical}, i.e. $H_t(\lambda z) = \lambda^2H_t(z)$ for every $\lambda\in\mathbb{R}$.

As was proved by Givental \cite{Giv}, or by the second author \cite{mio5} following Th\'{e}ret \cite{Th2}, the lift $\{\Phi_t\}_{t\in[0,1]}$ of a contact isotopy $\{\phi_t\}_{t\in[0,1]}$ of $\mathbb{R}P^{2n-1}$ has a 1-parameter family of generating functions $F_t: \mathbb{R}^{2n}\times\mathbb{R}^{2M} \rightarrow \mathbb{R}$, $t\in[0,1]$, which are \textit{conical} i.e. for each $F_t$ we have that $F_t(\lambda z,\lambda\zeta) = \lambda^2 F_t(z,\zeta)$ for every $\lambda\in\mathbb{R}$. Because of this property, the functions $F_t$ are determined by the induced functions $f_t: \mathbb{R}P^{2n+2M-1}\rightarrow\mathbb{R}$. These functions are useful to study the discriminant length of contact isotopies of $\mathbb{R}P^{2n-1}$ because of the following lemma (which plays a crucial role also in \cite{Giv} and \cite{mio5}).

\begin{lemma}\label{lemma: discriminant points}
For every $t\in[0,1]$, discriminant points of $\phi_t:\mathbb{R}P^{2n-1} \rightarrow \mathbb{R}P^{2n-1}$ are in 1-1 correspondence with critical points of $f_t: \mathbb{R}P^{2n+2M-1}\rightarrow\mathbb{R}$ with critical value $0$.
\end{lemma}

\begin{proof}
Given a contactomorphism $\phi$ of $\big(M,\xi=\text{ker}(\alpha)\big)$, every point of the symplectization $SM$ which is in the fiber above a discriminant point of $\phi$ is a fixed point of the lift $\Phi$. In particular, for every $t\in[0,1]$ we have a 1-1 correspondence between discriminant points of $\phi_t:\mathbb{R}P^{2n-1} \rightarrow \mathbb{R}P^{2n-1}$ and lines of fixed points of $\Phi_t: \mathbb{R}^{2n}\rightarrow\mathbb{R}^{2n}$. On the other hand, fixed points of $\Phi_t$ are in 1-1 correspondence with critical points of the generating function $F_t: \mathbb{R}^{2n}\times\mathbb{R}^{2M}\rightarrow\mathbb{R}$. Since the function $F_t$ is conical, critical points come in lines and have always critical value $0$. Moreover we have a 1-1 correspondence between lines of critical points of $F_t$ and critical points of $f_t$ with critical value $0$. Hence, critical points of $f_t$ of critical value $0$ are in 1-1 correspondence with discriminant points of $\phi_t$.
\end{proof}

In order to detect discriminant points and estimate the discriminant length, we will look at changes in the topology of the subsets $N_t:=\{f_t\leq0\}$ of $\mathbb{R}P^{2n+2M-1}$ for $t\in[0,1]$. As in \cite{Giv,Th2,mio5} the tool we use is the \emph{cohomological index} for subsets of projective spaces, that was introduced by Fadell and Rabinowitz \cite{FR}.

The cohomological index of a subset $X$ of a real projective space $\mathbb{R}P^m$ is defined as follows. Recall that $H^{\ast}(\mathbb{R}P^m;\mathbb{Z}_2)=\mathbb{Z}_2[u]/u^{m+1}$ where $u$ is the generator of $H^1(\mathbb{R}P^m;\mathbb{Z}_2)$. We define
$$
\text{ind}(X)=1+\text{max}\{\,k\in\mathbb{N}\;|\;i_X^{\phantom{X}\ast}(u^k)\neq 0\,\}
$$
where $i_X: X\hookrightarrow\mathbb{R}P^m$ is the inclusion (and set by definition $\text{ind}(\emptyset)=0$). In other words, $\text{ind}(X)$ is the dimension over $\mathbb{Z}_2$ of the image of the homomorphism $i_X^{\phantom{X}\ast}: H^{\ast}(\mathbb{R}P^m;\mathbb{Z}_2) \rightarrow H^{\ast}(X;\mathbb{Z}_2)$. Given a conical function $F: \mathbb{R}^{m}\rightarrow\mathbb{R}$ we denote by $\text{ind}(F)$ the index of $\{f\leq0\}\subset\mathbb{R}P^{m-1}$ where $f: \mathbb{R}P^{m-1}\rightarrow\mathbb{R}$ is the function on projective space induced by $F$. The following lemma was proved by Givental \cite[Appendices A and B]{Giv} (see also \cite[Lemma 5.2]{mio5}).

\begin{lemma}\label{lemma: giv}
Let $F$ and $G$ be conical functions defined on $\mathbb{R}^m$ and $\mathbb{R}^{m'}$ respectively, and consider the direct sum $F\oplus G: \mathbb{R}^{m+m'}\rightarrow\mathbb{R}$. Then we have
$$
\text{ind} \big(F\oplus G\big) = \text{ind}(F) + \text{ind}\big(G).
$$
\end{lemma}

Given a contact isotopy $\{\phi_t\}$ of $\mathbb{R}P^{2n-1}$ we define
$$\mu([\phi_t])=\text{ind}(F_0)-\text{ind}(F_1)$$
where $F_t: \mathbb{R}^{2n} \times \mathbb{R}^{2M} \rightarrow \mathbb{R}$ is a 1-parameter family of generating functions for the induced Hamiltonian isotopy $\Phi_t$ of $\mathbb{R}^{2n}$.

\begin{lemma}\label{lemma: well defined}
$\mu([\phi_t])$ is well-defined.
\end{lemma}

\begin{proof}
It was proved by Th\'{e}ret\footnote{Th\'{e}ret proved this for the lift to $\mathbb{R}^{2n}$ of a Hamiltonian isotopy of $\mathbb{C}P^{n-1}$. Exactly the same proof goes through also for the lift of a contact isotopy of $\mathbb{R}P^{2n-1}$, by replacing the $S^1$-symmetry by a $\mathbb{Z}_2$-symmetry.} \cite{Th2} that all 1-parameter families of conical generating functions $F_t$ associated to a fixed contact isotopy $\phi_t$ of $\mathbb{R}P^{2n-1}$ differ only by $\mathbb{Z}_2$- and $\mathbb{R}$-equivariant fiber preserving diffeomorphism and stabilization. Since a  $\mathbb{Z}_2$- and $\mathbb{R}$-equivariant diffeomorphism of $\mathbb{R}^{2n}\times\mathbb{R}^{2M}$ descends to a diffeomorphism of $\mathbb{R}P^{2n+2M-1}$, it does not affect the cohomological index of the sublevel sets of the generating functions. Regarding stabilization, recall that a family of functions $F'_t: \mathbb{R}^{2n}\times\mathbb{R}^{2M}\times\mathbb{R}^{2M'}\rightarrow\mathbb{R}$ is said to be obtained by stabilization from $F_t: \mathbb{R}^{2n}\times\mathbb{R}^{2M}\rightarrow\mathbb{R}$ if $F'_t = F_t \oplus Q$ where $Q: \mathbb{R}^{2M'}\rightarrow\mathbb{R}$ is a quadratic form. The invariance of $\mu([\phi_t])$ under stabilization of the generating functions follows thus from Lemma \ref{lemma: giv}. We have just shown that $\mu([\phi_t])$ does not depend on the choice of a 1-parameter family $F_t$ of generating functions for $\phi_t$. We now show that $\mu([\phi_t])$ does not depend on the choice of a representative of the homotopy class $[\phi_t]$. Let $\{\phi_t'\}_{t\in[0,1]}$ be another representative of $[\phi_t]$, and let $\{\phi_t^{\phantom{t}s}\}_{s\in[0,1]}$ be a homotopy with fixed endpoints joining $\{\phi_t^{\phantom{t}0}\}=\{\phi_t\}$ to $\{\phi_t^{\phantom{t}1}\}=\{\phi_t'\}$. Then we have a smooth 2-parameter family of functions $f_t^{\phantom{t}s}: \mathbb{R}P^{2n+2M-1}\rightarrow\mathbb{R}$ associated to the $\phi_t^{\phantom{t}s}$. In particular we get a 1-parameter family $f_1^{\phantom{t}s}$, $s\in[0,1]$ of functions associated to the same contactomorphism $\phi_1$. Arguing as in \cite[Lemma 4.8]{Th2} we see that there is a smooth isotopy $\Psi_s$, $s\in[0,1]$, of $\mathbb{R}P^{2n+2M-1}$ such that $f_1^{\phantom{1}s}\circ\Psi_s=f_1^{\phantom{1}0}$ for all $s$. So the sublevel sets $N_1^{\phantom{1}s}=\{\,f_1^{\phantom{t}s}\leq0\,\}$ are all diffeomorphic and hence their cohomological index is the same. The same argument also applies to show that all the $N_0^{\phantom{1}s}$ are diffeomorphic, and so in particular we see that $\mu([\phi_t])=\mu([\phi_t'])$.
\end{proof}

In \cite{Giv}, $\mu([\phi_t])$ is called the \textit{non-linear Maslov index} of the contact isotopy $\{\phi_t\}$ of $\mathbb{R}P^{2n-1}$. The non-linear Maslov index is a quasimorphism on the universal cover of the contactomorphism group of $\mathbb{R}P^{2n-1}$, as follows (see \cite{gabi}) from the next lemma. This lemma will be also needed later on, in our proof of unboundedness of the discriminant norm.

\begin{lemma}[\cite{Giv}, Theorem 9.1]\label{lemma: quasimorphism}
For every contact isotopy $\phi_t$ and contactomorphism $\psi$ of $\mathbb{R}P^{2n-1}$ we have that
$$
|\,\mu([\psi\circ\phi_t]) - \mu([\phi_t])\,| \leq 2n.
$$
\end{lemma}

The proof of this lemma follows from Lemma \ref{lemma: giv} and from the following two properties of generating functions and of the cohomological index.

\begin{enumerate}
\renewcommand{\labelenumi}{(\roman{enumi})}
\item Quasiadditivity of generating functions: Although the generating function for the composition $\psi\circ\phi_t$ is different\footnote{See the composition formula for example in \cite{Th2}.} than the direct sum of the generating functions for $\psi$ and $\phi_t$, it coincides with it on a subspace of codimension $2n$.
\item The Lefschetz property for the cohomological index: If $X'$ is a hyperplane section of $X\subset\mathbb{R}P^m$ then $\text{ind}(X') \geq \text{ind}(X)-1$.
\end{enumerate}
We refer to Givental \cite{Giv} for more details about this proof.

As we will now explain, by looking at how the non-linear Maslov index changes along a contact isotopy we can get information about the discriminant points at every time. Consider a contact isotopy $\{\phi_t\}_{t\in[0,1]}$ of $\mathbb{R}P^{2n-1}$. If $\phi_t$ does not have any discriminant point for all $t$ in a subinterval $(t_0,t_1]$ of $[0,1]$ then we must have $\mu\big([\{\phi_t\}_{t\in[0,t_0]}]\big) = \mu\big([\{\phi_t\}_{t\in[0,t_1]}]\big)$. Indeed, by Lemma \ref{lemma: discriminant points} we know that $0$ is a regular value of the corresponding functions $f_t: \mathbb{R}P^{2n+2M-1}\rightarrow\mathbb{R}$ for all $t\in(t_0,t_1]$, hence the sublevel sets $N_t\subset\mathbb{R}P^{2n+2M-1}$ are all diffeomorphic (see for example \cite[Lemma 2.17]{mio1}) and so in particular $\mu\big([\{\phi_t\}_{t\in[0,t_0]}]\big) = \mu\big([\{\phi_t\}_{t\in[0,t_1]}]\big)$. Suppose now that $\mu\big([\{\phi_t\}_{t\in[0,t_0]}]\big) \neq \mu\big([\{\phi_t\}_{t\in[0,t_1]}]\big)$ and that there is a single value of $t$ in $(t_0,t_1]$ for which $\phi_t$ belongs to the discriminant. Then we claim that the set of discriminant points of $\phi_t$ has index greater or equal than $|\mu\big([\{\phi_t\}_{t\in[0,t_0]}]\big) - \mu\big([\{\phi_t\}_{t\in[0,t_1]}]\big)|$. Since the set of discriminant points of $\phi_t$ is a subset of $\mathbb{R}P^{2n-1}$ its index is at most $2n$, and so it follows from our claim that
$|\mu\big([\{\phi_t\}_{t\in[0,t_0]}]\big) - \mu\big([\{\phi_t\}_{t\in[0,t_1]}]\big)| \leq 2n$. The claim can be seen as follows. It was proved by Th\'{e}ret \cite[Proposition 84]{Th1} (see also \cite[Lemma 5.3]{mio5}) that if $\mu\big([\{\phi_t\}_{t\in[0,t_0]}]\big) \neq \mu\big([\{\phi_t\}_{t\in[0,t_1]}]\big)$ and there is a single value of $t$ in $(t_0,t_1]$ for which $\phi_t$ belongs to the discriminant, then the set of critical points of $f_t$ with critical value $0$ has index greater or equal than $|\mu\big([\{\phi_t\}_{t\in[0,t_0]}]\big) - \mu\big([\{\phi_t\}_{t\in[0,t_1]}]\big)|$. As we have seen in Lemma \ref{lemma: discriminant points}, there is a bijection between the set of critical points of $f_t$ with critical value $0$ and the set of discriminant points of $\phi_t$. Still, it is not clear a priori that these two sets should have the same index, because they are contained in different projective spaces: the set of discriminant points of $\phi_t$ is contained in $\mathbb{R}P^{2n-1}$, while the set of critical points of $f_t$ with critical value $0$ is contained in $\mathbb{R}P^{2n+2M-1}$. However, the claim follows from the fact that the bijection described in Lemma \ref{lemma: discriminant points} is the restriction of a map $\underline{i}: \mathbb{R}P^{2n-1} \hookrightarrow \mathbb{R}P^{2n+2M-1}$ which is diffeomorphic to the standard inclusion $i: \mathbb{R}P^{2n-1} \hookrightarrow \mathbb{R}P^{2n+2M-1}$ (i.e. there is a diffeomorphism of $\mathbb{R}P^{2n+2M-1}$ that intertwines $i$ and $\underline{i}$). This fact can be seen by looking at the identifications underlying the construction of generating functions for Hamiltonian symplectomorphisms of $\mathbb{R}^{2n}$. Recall that a generating function for a Hamiltonian symplectomorphism $\Phi$ of $\mathbb{R}^{2n}$ is actually a generating function for the Lagrangian submanifold $\Gamma_{\Phi}$ of $T^{\ast}\mathbb{R}^{2n}$ that is the image of the graph of $\Phi$ under the identification $\tau: \overline{\mathbb{R}^{2n}} \times \mathbb{R}^{2n} \rightarrow T^{\ast}\mathbb{R}^{2n}$, $\tau(x,y,X,Y) = \big(\frac{x+X}{2},\frac{y+Y}{2},Y-y,x-X\big)$. Fixed points of $\Phi$ correspond to intersections of $\Gamma_{\Phi}$ with the 0-section. If $F: \mathbb{R}^{2n} \times \mathbb{R}^{2M} \rightarrow \mathbb{R}$ is a generating function for $\Phi$ then there is a diffeomorphism from the set of fiber critical points of $F$ to $\Gamma_{\Phi}$, given by the restriction of the map $\mathbb{R}^{2n} \times \mathbb{R}^{2M} \rightarrow T^{\ast}\mathbb{R}^{2n}$, $(q,\xi) \mapsto \big(q,\frac{\partial F}{\partial q} (q,\xi)\big)$. This diffeomorphism induces the bijection between critical points of $F$ and fixed points of $\Phi$ that appears in the proof of Lemma \ref{lemma: discriminant points}. But, if $\Phi$ is Hamiltonian isotopic to the identity (as it is in our case) then the set of fiber critical points of $F$ is diffeomorphic to $\mathbb{R}^{2n} \times \{0\} \subset \mathbb{R}^{2n} \times \mathbb{R}^{2M}$. This implies our claim.

We are now ready to prove that the $2k$-th iteration $\{\varphi_t\}_{t\in[0,1]}$ of the Reeb flow has discriminant length at least $k$. As was proved by Givental \cite{Giv} and Th\'{e}ret \cite{Th2}, we know that $\mu\big([\{\varphi_t\}_{t\in[0,1]}]\big)=4kn$. Let $\{\phi_t\}_{t\in[0,1]}$ be a contact isotopy representing $[\{\varphi_t\}_{t\in[0,1]}]$ which is of the form as described in Lemma \ref{lemma: decomposition} and minimizes the discriminant length. By Lemma \ref{lemma: well defined} we still have $\mu\big([\{\phi_t\}_{t\in[0,1]}]\big)=4kn$, hence if $k\neq0$ then, by the discussion above, there must be a value $t_0\in(0,1]$ such that $\phi_{t_0}$ belongs to the discriminant. Assume that $t_0$ is the smallest value of $t\in(0,1]$ for which $\phi_{t_0}$ belongs to the discriminant. Then, as discussed above, we must have that $\mu\big([\{\phi_t\}_{t\in[0,t_0]}]\big)\leq 2n$. Thus, we have shown that with a single embedded piece we can only get to a $t_0$ with $\mu\big([\{\phi_t\}_{t\in[0,t_0]}]\big)\leq 2n$. Write now $\{\phi_t\}_{t\in[0,1]}$ as the concatenation $\{\phi_t\}_{t\in[0,1]} = \{\phi_t\}_{t\in[0,t_0]} \ast \{\phi_t\}_{t\in[t_0,1]}$. By definition we have that $\mu\big(\{\phi_t\}_{t\in[0,1]}\big) = \mu\big( \{\phi_t\}_{t\in[0,t_0]}\big) + \mu\big(\{\phi_t\}_{t\in[t_0,1]} \big)$ hence $\mu\big(\{\phi_t\}_{t\in[t_0,1]} \big) \geq 2n\,(2k-1)$. For $t\in(t_0,1]$ we are not interested anymore in detecting values of $t$ for which $\phi_t$ is in the discriminant, but instead we want to detect values of $t$ for which $\phi_{t_0}^{\phantom{t_0}-1}\circ\phi_t$ is in the discriminant and thus $\text{gr}(\phi_{t_0})$ and $\text{gr}(\phi_t)$ intersect. Write thus $\{\phi_t\}_{t\in[t_0,1]} = \{\phi_{t_0}\circ(\phi_{t_0}^{\phantom{t_0}-1}\circ\phi_t)\}_{t\in[t_0,1]}$. By Lemma \ref{lemma: quasimorphism} we have that $\mu\big(\{\phi_{t_0}\circ(\phi_{t_0}^{\phantom{t_0}-1}\circ\phi_t)\}_{t\in[t_0,1]} \big) - \mu\big(\{\phi_{t_0}^{\phantom{t_0}-1}\circ\phi_t\}_{t\in[t_0,1]} \big) \leq 2n$, and so $\mu\big(\{\phi_{t_0}^{\phantom{t_0}-1}\circ\phi_
t\}_{t\in[t_0,1]} \big) \geq 2n(2k-1) - 2n = 2n (2k-2)$. If $k>1$ then, by the same argument as before, there must be some $t_1$ in $(t_0,1]$ such that $\phi_{t_0}^{\phantom{t_0}-1}\circ\phi_{t_1}$ belongs to the discriminant, and hence $\text{gr}(\phi_{t_0})$ and $\text{gr}(\phi_{t_1})$ intersect. We continue in this way and get that the discriminant length of $\{\phi_t\}_{t\in[0,1]}$ is at least $k$.

\begin{rmk}
It would be interesting to understand whether this estimate is sharp, i.e. whether there is a contact isotopy of length $k$ in the same homotopy class of the $2k$-th iteration of the Reeb flow.
\end{rmk}

Unboundedness of the discriminant oscillation norm is proved by combining the above argument with the monotonicity of the non-linear Maslov index, which is described in the next lemma.

\begin{lemma}\label{lemma: monotonicity of non-linear Maslov index}
If $\{\phi_t\}$ is a positive (respectively negative) contact isotopy of $\mathbb{R}P^{2n-1}$ then $\mu\big([\{\phi_t\}]\big) \geq 0$ (respectively $\mu\big([\{\phi_t\}]\big) \leq 0$).
\end{lemma}

\begin{proof}
As it is proved for example in \cite[Lemma 3.6]{mio5}, if $\{\phi_t\}$ is a positive contact isotopy of $\mathbb{R}P^{2n-1}$ then there is a 1-parameter family of generating functions $F_t: \mathbb{R}^{2n} \times \mathbb{R}^{2M} \rightarrow \mathbb{R}$ which is increasing, i.e. $\frac{\partial F_t}{\partial t}(q,\xi) > 0$ for all $(q,\xi)\in \mathbb{R}^{2n} \times \mathbb{R}^{2M}$. Hence $\text{ind}(F_t)$ is decreasing in $t$ and so $\mu\big([\{\phi_t\}]\big) = \text{ind}(F_0) - \text{ind}(F_1)\geq 0$.
\end{proof}

Note that if $\{\phi_t\}$ is a positive loop of contactomorphisms then $\mu\big([\{\phi_t\}]\big) > 0$. On the other hand, if $\{\phi_t\}$ is contractible then $\mu\big([\{\phi_t\}]\big) = \mu\big([\text{id}]\big) = 0$. As noticed by Eliashberg and Polterovich \cite{EP}, this shows that there are no positive contractible loops of contactomorphisms of $\mathbb{R}P^{2n-1}$, i.e. that $\mathbb{R}P^{2n-1}$ is orderable. Hence the discriminant oscillation norm is non-degenerate.

\section{The Legendrian discriminant length}\label{section: Legendrian}

Let $\big(M,\xi=\text{ker}(\alpha)\big)$ be a (cooriented) contact manifold. We will now define the discriminant length of a Legendrian isotopy in $M$. We first give the Legendrian analogue of Lemma \ref{lemma: decomposition}.

\begin{lemma}\label{lemma: decomposition legendrian}
Let $\{L_t\}_{t\in[0,1]}$ be a Legendrian isotopy in $M$. After perturbing $\{L_t\}$ in the same homotopy class with fixed endpoints, there exist an integer $N$ and a subdivision $0=t_0<t_1<\cdots<t_{N-1}<t_N=1$ such that for all $i=0,\cdots,N-1$ the submanifold $\bigcup_{t\in[t_i,t_{i+1}]}L_t$ of $M$ is embedded.
\end{lemma}

\begin{proof}
Let $\{\phi_t\}_{t\in[0,1]}$ be a contact isotopy of $M$ such that $L_t=\phi_t(L_0)$ for all $t$ (such a contact isotopy exists because of the Legendrian isotopy extension theorem, see \cite{gei}). As we saw in the proof of Lemma \ref{lemma: decomposition}, $\{\phi_t\}$ is homotopic with fixed endpoints to the concatenation of $\{\varphi_t\circ\phi_t\}$ and $\{\varphi_t^{\phantom{t}-1}\circ(\varphi_1\circ\phi_1)\}$ for every other contact isotopy $\{\varphi_t\}$. Hence $\{L_t\}$ is homotopic with fixed endpoints to the concatenation of $\{\varphi_t\circ\phi_t(L_0)\}$ and $\{\varphi_t^{\phantom{t}-1}\big(\varphi_1\circ\phi_1(L_0)\big)\}$. If $\{\varphi_t\}$ is generated by a sufficiently big contact Hamiltonian then $\{\varphi_t\circ\phi_t(L_0)\}$ is positive and $\{\varphi_t^{\phantom{t}-1}\big(\varphi_1\circ\phi_1(L_0)\big)\}$ is negative. Thus it is enough to show that if $\{L_t\}_{t\in[0,1]}$ is a positive (or negative) Legendrian isotopy which is sufficiently $\mathcal{C}^1$-small then it is embedded. This can be proved exactly as for Lemma \ref{lemma: decomposition}, by using Weinstein's theorem and the Hamilton-Jacobi equation.
\end{proof}

\begin{defi}
The \textbf{discriminant length} of the homotopy class of a Legendrian isotopy $\{L_t\}$ is the minimal integer $N$ needed to represent it as described in Lemma \ref{lemma: decomposition legendrian}. We also set by definition the discriminant length of the homotopy class of a constant Legendrian isotopy to be zero.
\end{defi}

If $\{L_t\}_{t\in [0,1]}$ is a Legendrian isotopy of $M$ which is already of the form as described in Lemma \ref{lemma: decomposition legendrian}, we will call the \textit{discriminant length} of $\{L_t\}$ the minimal number $N$ for which there exists a subdivision $0=t_0<t_1<\cdots<t_{N-1}<t_N=1$ such that for all $i=0,\cdots,N-1$ the submanifold $\bigcup_{t\in[t_i,t_{i+1}]}L_t$ is embedded. The discriminant length of the homotopy class of $\{L_t\}$ is then the minimal discriminant length of a representative which is of the form as described in Lemma \ref{lemma: decomposition legendrian}.

The zigzag and oscillation lengths of the homotopy class of a Legendrian isotopy are defined by modifying the definition of the discriminant length in the same way as for the case of contact isotopies. As we will now show, the discriminant and zigzag lengths on the universal cover $\widetilde{\mathcal{L}}(M,\xi)$ of the space of Legendrians are equivalent.

\begin{prop}\label{proposition: equivalence legendrian}
For every Legendrian isotopy $\{L_t\}_{t\in[0,1]}$ the zigzag length of its homotopy class is smaller than or equal to twice the discriminant length.
\end{prop}

\begin{proof}
It is enough to show that if $\{L_t\}_{t\in[0,1]}$ is an embedded Legendrian isotopy then we can deform it (in the same homotopy class with fixed endpoints) into an embedded zigzag, i.e. into the concatenation of a positive and a negative embedded Legendrian isotopies. The construction of the embedded zigzag goes as follows. Consider first a positive Legendrian isotopy $\{L'_t\}_{t\in[0,\delta]}$ obtained by pushing $L_0$ by the Reeb flow for time $t\in[0,\delta]$. Then $L'_0=L_0$ and, for $\delta$ small enough, $\{L'_t\}_{t\in[0,\delta]}$ is embedded. Note that if $\epsilon$ is small enough then $\{L_t\}_{t\in[0,\epsilon]}$ does not intersect $L'_{\delta}$. Take now a contactomorphism $\phi$ of $M$ such that for all $t\in[0,\epsilon]$ we have $\phi(L_t) = L_{\frac{t}{\epsilon}}$. Hence, $\phi$ sends the whole isotopy $\{L_t\}_{t\in[0,\epsilon]}$ to $\{L_t\}_{t\in[0,1]}$ without moving $L_0$. Such a contactomorphism $\phi$ can be found by applying to the family of Legendrian isotopies $\{L_t\}_{t\in[0,s]}$ for $s\in[\epsilon,1]$ a 1-parameter version of the Legendrian isotopy extension theorem: given a family of embedded Legendrian isotopies $\{L_t^{\phantom{t}s}\}_{t\in[0,1]}$ for $s\in[0,1]$, there is a contact isotopy $\{\phi_s\}$ of $M$ such that, for all $s$, $\phi_s(L_t^{\phantom{t}0}) = L_t^{\phantom{t}s}$ for all $t$. We take the positive embedded Legendrian isotopy $\{\phi(L'_t)\}_{t\in[0,\delta]}$ to be the first part of our zigzag. Note that the endpoint $\phi(L'_{\delta})$ does not intersect $\bigcup_{t\in[0,1]}L_t$. Indeed, $\bigcup_{t\in[0,1]}L_t = \phi\big( \bigcup_{t\in[0,\epsilon]}L_t\big)$ and we know that $L'_{\delta}$ does not intersect $\bigcup_{t\in[0,\epsilon]}L_t$. Consider now a contact isotopy $\psi_t$ of $M$ such that $\psi_t(L_0)=L_t$ for all $t\in[0,1]$. Since $\phi(L'_{\delta})$ does not intersect $\bigcup_{t\in[0,1]}L_t$ we can assume that $\phi(L'_{\delta})$ does not belong to the support of the isotopy $\{\psi_t\}$. The inverse of $\{\psi_1\big(\phi(L'_t)\big)\}_{t\in[0,\delta]}$ is a negative embedded Legendrian isotopy connecting $\phi_1(L'_{\delta})$ to $L_1$. It is the second part of our zigzag. The zigzag we just constructed is in the same homotopy class as the initial $\{L_t\}$. An explicit deformation between the two is given as follows. First keep the first piece $\{\phi(L'_t)\}$ fixed, and deform the second piece with a parameter $s$ decreasing from $1$ to $0$ by replacing it at time $s$ by the inverse of $\psi_s\big(\phi(L'_t)\big)$ followed by $\{L_t\}_{t\in[s,1]}$. When $s=1$ we obtain thus the concatenation of $\{\phi(L'_t)\}$, its inverse, and $\{L_t\}_{t\in[0,1]}$. We can then homotope the concatenation of $\{\phi(L'_t)\}$ and its inverse to the constant isotopy $\{L_0\}$.
\end{proof}

Let $\{\phi_t\}_{t\in[0,1]}$ be a contact isotopy of $M$ which is of the form as described in Lemma \ref{lemma: decomposition}. Then its discriminant length is equal to the discriminant length of the Legendrian isotopy $\{\text{gr}(\phi_t)\}_{t\in[0,1]}$ of $M\times M\times\mathbb{R}$. Note however that the discriminant length of the homotopy class of $\{\phi_t\}$ is not necessarily equal to Legendrian discriminant length  of the homotopy class of $\{\text{gr}(\phi_t)\}$ since there could be a shorter Legendrian isotopy which is homotopic to $\{\text{gr}(\phi_t)\}$ but which is not the graph of a contact isotopy, or is not the graph of a contact isotopy which is homotopic to $\{\phi_t\}$. Thus the map
$$
j:\,\widetilde{\text{Cont}_0}(M,\xi) \rightarrow\,\widetilde{\mathcal{L}_{\Delta}}(M\times M \times\mathbb{R})
$$
$$
[\{\phi_t\}] \mapsto [\{\text{gr}(\phi_t)\}]
$$
a priori does not necessarily preserve the discriminant length. It seems plausible that there might be examples of contact manifolds $M$ for which $j$ does indeed not preserve the discriminant length. Note that this question was studied by Ostrover \cite{O} in the context of Hofer geometry. Ostrover proved that for any closed symplectic manifold $W$ with $\pi_2(W)=0$ the map $j: \phi \mapsto \text{gr}(\phi)$ from $\text{Ham}(W)$ to the space $\mathcal{L}_{\Delta}(W\times W)$ of all Lagrangians that are exact Lagrangian isotopic to the diagonal does not preserve the Hofer distance. Moreover he showed that the image of $\text{Ham}(W)$ inside $\mathcal{L}_{\Delta}(W\times W)$ is \textquotedblleft strongly distorted\textquotedblright. It would be interesting to understand if a similar phenomenon also appears in some cases  for the discriminant (zigzag, oscillation) metric.

In the rest of this section we will prove unboundedness of the Legendrian discriminant length in three special cases. The arguments can be adapted as in Sections \ref{section: RxS} and \ref{section: projective space} to show unboundedness also for the oscillation length.

\subsection{Unboundedness of the Legendrian discriminant length for $T^{\ast}B\times S^1$}\label{subsection: TBxS}

We will show that the Legendrian discriminant length is unbounded in $T^{\ast}B\times S^1$, for every smooth closed manifold $B$. More precisely we will show that the Legendrian isotopy of $T^{\ast}B\times S^1$ given by the image of the 0-section by the $k$-th iteration of the Reeb flow has discriminant length equal to $k$. Recall that the Reeb flow on $T^{\ast}B\times S^1$ is given by rotation in the $S^1$-direction.

Note that every Legendrian isotopy $\{L_t\}$ of $T^{\ast}B\times S^1$ can be uniquely lifted (once we choose a starting point) to a Legendrian isotopy $\{\widetilde{L_t}\}$ of $J^1B=T^{\ast}B\times \mathbb{R}$. For example, for the Legendrian isotopy given by the image of the 0-section by the $k$-th iteration of the Reeb flow we will consider the lift $\widetilde{L_t} = 0_B \times \{kt\}$. To calculate the discriminant length of $\{L_t\}$ we will use the spectral invariants for Legendrian submanifolds of $J^1B$, and an argument similar to the one we gave in Section \ref{section: RxS} to show that the discriminant metric in $\widetilde{\text{Cont}^{\phantom{0}c}_0}(\mathbb{R}^{2n}\times S^1)$ is unbounded.

Recall first that for every Legendrian submanifold $L$ of $J^1B$ which is isotopic to the 0-section, and for any cohomology class $u\in H^{\ast}(B)$ we can define a spectral number $c(u,L)\in\mathbb{R}$ by applying a minimax method to a generating function quadratic at infinity for $L$.  We will use the following properties (see \cite{mio1,mio2}) of these spectral numbers.

\begin{lemma}\label{lemma: properties spectral invariants legendrians} The spectral numbers $c(u,L)$ for Legendrian submanifolds $L$ of $J^1B$ satisfy the following properties.
\begin{enumerate}
\renewcommand{\labelenumi}{(\roman{enumi})}
\item For any $u\in H^{\ast}(B)$, a Legendrian submanifold $L$ which is isotopic to the 0-section intersects the 0-wall at a point of the form $\big(q,0,c(u,L)\big)$ for some $q\in B$. As a consequence, for every $u\in H^{\ast}(B)$ and $\lambda\in\mathbb{R}$ we have that  $c\big(u,0_B\times\{\lambda\}\big) = \lambda$.
\item $c(u\cup v, L_1+L_2)\geq c(u,L_1)+c(v,L_2)$, where $L_1+L_2$ is defined by\footnote{Although $L_1 + L_2$ is not necessarily a submanifold one can still define $c(u\cup v, L_1+L_2)$, see for example the comment in \cite[Lemma 2.1]{mio2}.}
$$ L_1+L_2:=\{\;(q,p,z)\in J^1B \;|\; p=p_1+p_2, \; z=z_1+z_2, \;(q,p_1,z_1)\in L_1, \; (q,p_2,z_2)\in L_2  \;\}.$$
\item $c(\mu,\overline{L})=-c(1,L)$ where $\mu$ and $1$ denote respectively the volume and unit class in $H^{\ast}(B)$ and where $\overline{L}$ denotes the image of $L$ under the map $J^1B\rightarrow J^1B$, $(q,p,z)\mapsto(q,-p,-z)$.
\item For any Legendrian isotopy $\{L_t\}$, $c(u,L_t)$ is continuous in $t$.
\item If $\Psi$ is a contactomorphism of $J^1B$ which is 1-periodic in the $\mathbb{R}$-coordinate and is isotopic to the identity through 1-periodic contactomorphisms then $$\lceil c(u,\Psi(L))\rceil=\lceil c\big(u,L-\Psi^{-1}(0_B)\big)\rceil.$$
\item If $\{L_t\}$ is a positive (respectively negative) Legendrian isotopy then the function $t \mapsto c(u,L_t)$ is increasing (respectively decreasing).
\end{enumerate}
\end{lemma}

We will write $c^+(L):=c(\mu,L)$ and $c^-(L):=c(1,L)$.

Let now $\{L_t\}_{t\in[0,1]}$ be the Legendrian isotopy of $T^{\ast}B\times S^1$ which is the image of the 0-section by the $k$-th iteration of the Reeb flow. Consider the lift given by $\widetilde{L_t} = 0_B \times \{kt\}$ for $t\in[0,1]$. We will show that the discriminant length of $\{L_t\}_{t\in[0,1]}$ is equal to $k$.

Note first that the discriminant length of $\{L_t\}_{t\in[0,1]}$ cannot be $1$. Indeed, since $c^+(\widetilde{L_0})=0$ and $c^+(\widetilde{L_1})=k$, by Lemma \ref{lemma: properties spectral invariants legendrians}(iv) there must be a $t_0\in[0,1]$ such that $c^+(\widetilde{L_{t_0}})=1$. But then Lemma \ref{lemma: properties spectral invariants legendrians}(i) implies that $L_{t_0}$ must intersect the 0-section, and hence $\{L_t\}_{t\in[0,1]}$ is not embedded. Thus, with a first embedded piece we can at most reach a $t_0$ with $\lceil c^+(L_{t_0})\rceil = 1$.  We will now show that with another embedded piece we cannot get any further than getting at most $c^+(\widetilde{L_{t}})=2$. Indeed, assume that with a second embedded piece we can get to a $t>t_0$ with $\lceil c^+(\widetilde{L_{t}})\rceil=3$. We claim that $\lceil c^+(\widetilde{L_{t}})\rceil \leq \lceil c^+(\widetilde{L_{t_0}})\rceil + \lceil c^+\big(\tilde{\Psi}_{t_0}^{\phantom{t}-1}(\widetilde{L_{t}})\big)\rceil$ where $\tilde{\Psi}_{t_0}$ is the lift to $J^1B$ of a contactomorphism of $T^{\ast}B\times S^1$ that sends the 0-section to $L_{t_0}$. Indeed, by Lemma \ref{lemma: properties spectral invariants legendrians}(v) we have $\lceil c^+\big(\tilde{\Psi}_{t_0}^{\phantom{t}-1}(\widetilde{L_{t}})\big)\rceil = \lceil c^+\big(\widetilde{L_{t}} - \tilde{\Psi}_{t_0}(0_B)\big)\rceil$. But, by Lemma \ref{lemma: properties spectral invariants legendrians}(ii) and (iii),
$$c^+\big(\widetilde{L_{t}} - \tilde{\Psi}_{t_0}(0_B)\big) = c^+\big(\widetilde{L_{t}} - \widetilde{L_{t_0}}\big) = c\big(\mu\cup 1, \widetilde{L_{t}} + \overline{\tilde{L_{t_0}}}\big) \geq c(\mu, \widetilde{L_{t}}) + c(1,\overline{\tilde{L_{t_0}}}) = c^+(\widetilde{L_{t}}) - c^+(\widetilde{L_{t_0}})$$
and so $\lceil c^+\big(\tilde{\Psi}_{t_0}^{\phantom{t}-1}(\widetilde{L_{t}})\big)\rceil \geq \lceil c^+(\widetilde{L_{t}})\rceil - \lceil c^+(\widetilde{L_{t_0}})\rceil$ as we wanted. Thus we get $\lceil c^+\big(\tilde{\Psi}_{t_0}^{\phantom{t}-1}(\widetilde{L_{t}})\big)\rceil \geq 2$ and so we see that there must be a $t_1$ in $[t_0,t]$ for which $\lceil c^+\big(\tilde{\Psi}_{t_0}^{\phantom{t}-1}(\widetilde{L_{t_1}})\big)\rceil =1$. We then have that $\Psi_{t_0}^{\phantom{t}-1}(L_{t_1})$ intersects the 0-section. But this is equivalent to saying that $L_{t_1}$ intersects $L_{t_0}$, which is a contradiction. Continuing with this argument we see that the discriminant length of $\{L_t\}_{t\in[0,1]}$ is equal to $k$.

\subsection{Unboundedness of the Legendrian discriminant length for $\mathbb{R}P^{2n-1}$}\label{subsection: projective space}

Using the same techniques as in Section \ref{section: projective space} we can also prove that the Legendrian discriminant length in $\mathbb{R}P^{2n-1}$ is unbounded. We see $\mathbb{R}P^{2n-1}$ as the projectivisation of $\mathbb{R}^{2n}$, and denote by $\pi: \mathbb{R}^{2n}\setminus\{0\}\rightarrow\mathbb{R}P^{2n-1}$ the projection. If $L$ is a Legendrian submanifold of $\mathbb{R}P^{2n-1}$ then $\widetilde{L}:= \pi^{-1}(L)$ is a (conical) Lagrangian submanifold of $\mathbb{R}^{2n}$. Note that we can identify $\mathbb{R}^{2n}$ with $T^{\ast}\mathbb{R}^{n}$, by regarding the first $n$ components of $\mathbb{R}^{2n}$ as the 0-section of $T^{\ast}\mathbb{R}^{n}$. In this way we can associate to Legendrian submanifolds of $\mathbb{R}P^{2n-1}$ the generating function of their lift to $\mathbb{R}^{2n}$, seen as a Lagrangian submanifold of $T^{\ast}\mathbb{R}^{n}$. More precisely, consider a Legendrian isotopy $\{L_t\}_{t\in[0,1]}$ in $\mathbb{R}P^{2n-1}$, starting at the Legendrian submanifold $L_0$ of $\mathbb{R}P^{2n-1}$ that corresponds to the 0-section of $T^{\ast}\mathbb{R}^{n}$. Then its lift $\widetilde{L_t}$ to $\mathbb{R}^{2n}\equiv T^{\ast}\mathbb{R}^{n}$ has a 1-parameter family of conical generating functions $F_t: \mathbb{R}^n\times\mathbb{R}^N\rightarrow\mathbb{R}$, $t\in[0,1]$. Being conical, these functions are determined by the induced functions $f_t: \mathbb{R}P^{n+N-1}\rightarrow\mathbb{R}$, $t\in[0,1]$. Note that critical points of $f_t$ with critical value $0$ correspond to intersections of $L_t$ with $L_0$. As in Section \ref{section: projective space}, we can study the discriminant length of a Legendrian isotopy $\{L_t\}_{t\in[0,1]}$ by looking at the variation of the non-linear Maslov index $\mu([L_t]):= \text{ind}(F_0) - \text{ind}(F_1)$. The key Lemma \ref{lemma: quasimorphism} holds (with the same proof) also in this context: if $\{L_t\}_{t\in[0,1]}$ is a Legendrian isotopy and $\phi$ a contactomorphism then $|\,\mu\big([\phi(L_t)]\big) - \mu([L_t])\,| \leq n$. Consider now the Legendrian isotopy $\{L_t=\varphi_t(L_0)\}$ where $L_0$ is the Legendrian submanifold of $\mathbb{R}P^{2n-1}$ corresponding to the 0-section of $T^{\ast}\mathbb{R}^{n}$ and where $\{\varphi_t\}$ is the $k$-th iteration of the Reeb flow. Then $\mu([L_t]) = 2nk$ (see Givental \cite{Giv}). Arguing along the same lines as in Section \ref{section: projective space} we see thus that the discriminant length of $L_t$ is at least $k$. Note that monotonicity of the non-linear Maslov index also holds for Legendrian isotopies, and can be used to prove that the Legendrian oscillation length is also unbounded.

\subsection{Other examples}\label{subsection: other examples}

In the case where $B=S^1$, the contact manifold $(T^* B\times S^1 ,\ker (dz-pdq ) )$, $(q,p,z)\in \R / \Z \times \R \times \R /\Z$ is $$(T^2 \times (-\pi /2 ,\pi /2 ) ,\zeta =\ker (\sin \theta dx +\cos \theta dy ))$$
where $(x,y,\theta)\in \R /\Z \times \R /\Z \times (-\pi /2 ,\pi /2 )$, {\it via} the change of variables:
$$q=-x, \quad p=\tan \theta , \quad z=y.$$ By the discussion in \ref{subsection: TBxS} we know that  the homotopy class of the path given by $L_t =\{ (x,tk,0)\}$ for $t\in [0,1]$ is of discriminant length equal to $k$.

One can use this result to obtain Legendrian paths of arbitrary discriminant (or oscillation) length in various examples of closed $3$-manifolds. Here is one of them.

Let $\Sigma$ be a closed oriented surface of genus $g\geq 1$ and $M$ be a $S^1$-fibration $\pi :M\to \Sigma$ over $\Sigma$, together with a contact structure $\xi$ that is $S^1$-invariant and transverse to the fibers. For any embedded curve $\gamma \subset \Sigma$, the incompressible torus $T_\gamma  =\pi^{-1} (\gamma )$ has a non singular characteristic foliation $\xi T_\gamma$. This foliation is linear and its slope depends continuously on the curve $\gamma$. By the contact condition, if we move $\gamma$ in one direction the slope changes and thus, by moving $\gamma$ a little bit, one can assume  that $\xi T_\gamma$ is a foliation by Legendrian circles $(L_s )_{s\in \R /\Z}$.

\begin{prop} For all $k\in \N$, the discriminant (and oscillation) length of the homotopy class of the Legendrian isotopy $(L_s )_{s\in [0,k]}$ is at least $k$.
\end{prop}

\begin{proof}  We consider the infinite cover $\widehat{M}$ of group $\pi_1 (T_\gamma )$ of $M$. 
It is the pull-back $\widehat{\pi} :\widehat{M} \rightarrow \widehat{\Sigma}$ of the fibration $\pi$ along an infinite cover $\widehat{\Sigma} \rightarrow \Sigma$, where $\widehat{\Sigma}$ is diffeomorphic to $\gamma \times \R$. The manifold $\widehat{M}$ is thus diffeomorphic to $T^2 \times \R$, with $T_\gamma \simeq T^2 \times \{0 \}$. The pullback $\widehat{\xi}$ of $\xi$ in $\widehat{M}$ is a connection for the $S^1$-fibration $\widehat{M} \rightarrow \widehat{\Sigma}$.
Taking coordinates in the base $\widehat{\Sigma}$ extended to $\widehat{M}$ by coordinates in the $S^1$-fiber direction, we easily get that $(\widehat{M} ,\widehat{\xi} )$  can be injectively immersed in
$(T^2 \times (-\pi /2,\pi/2 ) ,\ker (\sin \theta dx +\cos \theta dy ))$, by a map taking $T$ to  $T^2 \times \{0 \}$.
Moreover any isotopy of $L_0$ admits a lift in  $\widehat{M}$ whose projection to $M$ is one-to-one on its image.
We thus can deduce the proof of the proposition from the result on $T^* S^1 \times S^1$.
\end{proof}

\end{document}